\documentclass[a4paper,12pt]{amsart}
\usepackage{amsmath,latexsym,amssymb,amsfonts}
\usepackage{amscd,amssymb,epsfig}
\usepackage[arrow, matrix, curve]{xy}
\usepackage{hyperref}
\usepackage{mathdots}
\usepackage{color}

\numberwithin{equation}{section}

\newtheorem{theorem}{Theorem}[section]
\newtheorem{lemma}[theorem]{Lemma}

\newtheorem{proposition}[theorem]{Proposition}

\theoremstyle{definition}

\newtheorem{remark}[theorem]{Remark}

\renewcommand{\epsilon}{\varepsilon}

\newcommand{\B}{\mathcal{B}}
\newcommand{\C}{\mathbb{C}}

\newcommand{\E}{\mathbb{E}}

\newcommand{\M}{\mathcal{M}}
\newcommand{\N}{\mathcal{N}}

\newcommand{\R}{R}

\newcommand{\uX}{\mathbf{x}}

\newcommand{\uY}{\mathbf{y}}
\newcommand{\uZ}{\mathbf{z}}

\newcommand{\ux}{\mathbf{x}}

\newcommand{\diff}{\mathrm{d}}

\newcommand{\id}{\operatorname{id}}

\newcommand{\tr}{\operatorname{tr}}
\newcommand{\Tr}{\operatorname{Tr}}

\renewcommand{\Im}{\operatorname{Im}}

\makeatletter
\def\moverlay{\mathpalette\mov@rlay}
\def\mov@rlay#1#2{\leavevmode\vtop{%
\baselineskip\z@skip \lineskiplimit-\maxdimen
\ialign{\hfil$#1##$\hfil\cr#2\crcr}}}
\makeatother

\makeatletter
\def\@settitle{\begin{center}%
  \baselineskip14\p@\relax
    \normalfont \Large \uppercase{ \textbf{\@title}}
  \end{center}%
 }

\makeindex

\title[Operator-Valued Matrices]{{Operator-Valued Matrices with Free \\ or Exchangeable Entries}}

{\small\author[M. Banna]{Marwa Banna}}
\address{New York University Abu Dhabi, Division of Science, Mathematics, Abu Dhabi, UAE}
\email{marwa.banna@nyu.edu}

\author[G. C\'ebron]{Guillaume C\'ebron}
\address{Universit\'e Paul Sabatier, Laboratoire de Statistique et Probabilit\'es, 31062 Toulouse, France}
\email{Guillaume.Cebron@math.univ-toulouse.fr}

\date{\today}

\thanks{MB was partially supported by the ERC Advanced Grant NCDFP 339760 held by Roland Speicher. GC is supported by the Project
MESA (ANR-18-CE40-006) and by the Project STARS (ANR-20-CE40-0008) of the French National Research Agency (ANR). The authors would  like to thank the LabEx CIMI for covering some traveling expenses to work on this paper. The authors would also like to thank Tobias Mai for fruitful discussions on the Lindeberg method  and the referee for the comments/suggestions that helped improve the paper.}



\begin{document}

\begin{abstract}
We study matrices whose entries are free or exchangeable noncommutative elements in some tracial  $W^*$-probability space. More precisely, we consider operator-valued Wigner and Wishart matrices and prove quantitative convergence to operator-valued semicircular elements over some subalgebra in terms of Cauchy transforms and the Kolmogorov distance.  As direct applications, we obtain explicit rates of convergence  for a large class of  random block matrices with independent or correlated blocks. Our approach relies  on a noncommutative extension of the Lindeberg method and operator-valued Gaussian  interpolation techniques. 
\end{abstract}

\maketitle

\section{Introduction}

One of the first and  fundamental results in random matrix theory is that  the empirical spectral distribution of a Wigner matrix whose entries are  i.i.d. random variables converges in distribution to the semicircular law.  When considering matrices with exchangeable entries, Chatterjee \cite{Ch-06} showed that the limiting distribution is again semicircular. The main contribution of this paper is an extension of the latter result to operator-valued matrices with exchangeable entries. An example of such matrices are some random block matrices for they can be seen as matrices with entries in some  noncommutative algebra. For instance, if the matrix blocks are  i.i.d.  then they are exchangeable in the noncommutative space of random matrices and fit nicely in the framework of operator-valued matrices with exchangeable entries.

 More precisely, the entries of the  considered matrices consist of a \emph{finite}  family of exchangeable elements in some finite-dimensional $W^*$-probability space $(\M , \tau)$ whose distribution with respect to $\tau$ is invariant under   permutations.  
 We prove that the  distribution of operator-valued Wigner and Wishart matrices is close to that of an operator-valued semicircular element over some subalgebra and give precise \emph{quantitative estimates} for the associated Cauchy transforms. Whenever the limiting distribution is H\"older continuous, these estimates can be passed immediately onto the Kolmogorov distance.
 
Immediate consequences are {quantitative estimates} of Cauchy transforms, and in some cases on the Kolmogorov distance, of random block matrices with independent blocks, or equivalently, random matrices in which the  blocks  are themselves correlated but have  i.i.d. entries. These models gained lots of attention and were investigated in \cite{Girko2000, Thorbjornsen00,Oraby2007,Bryc2008,Ra-Or-Br-Sp-08,Diaz, An-Er-Kr, Erdos-al-19, Alt-al-Kronecker-19}. 

A main ingredient in our proof concerns an independence structure hidden behind exchangeability.  In fact, we prove in Section \ref{section:exch}  that  sums of exchangeable operators are close  in distribution to the expectation of sums of  independent averaged Gaussian operators. This first step of randomization, which consists of replacing a family of operators by a family of random operators, is new in the noncommutative setting, as it involves the structure of noncommutative probability spaces together  with the classical notion of randomness. It might  be surprising or seem unnatural but it is a consequence of the fact that invariance under permutations is a commutative concept; contrary to other notions of invariance, such as quantum exchangeability.
The distribution of matrices in random operators, is then shown to be close to that of an operator-valued semicircular element.

Another instance of operator-valued matrices with exchangeable entries is when the entries are identically distributed and freely independent, in the sense of Voiculescu, or more generally satisfying a certain noncommutative notion of independence, as studied in \cite{Voiculescu90,Shlyakhtenko97,Ryan-98,Ni-Sh-Sp-02,Popa-Hao-17a,Popa-Hao-17b,Liu2018}. In Sections \ref{section:wignerfree} and \ref{Section:Wishart-free}, we also prove  that the  distributions of  Wigner and Wishart matrices in  free entries are close to that of  operator-valued semicircular elements over some subalgebra of matrices. In the case of free entries, we can even go beyond this and allow different distributions of the entries, thus extending the above cited results to matrices with variance profile. The  provided  explicit quantitative estimates for the associated Cauchy transforms can be passed in some cases  to the Kolmogorov distance.

In both cases, our approximation techniques rely on the Lindeberg method which we extend to Cauchy transforms in the noncommutative setting.  This method is also known in probability theory as the replacement trick and has been first introduced by Lindeberg \cite{Lindeberg} to give an  elegant alternative and illustrative proof of the CLT of sums of independent random variables. It has later become widely applied in various problems and specially for establishing explicit rates of convergence. Contributions and refinements were later carried on this method to extend it to more general functions than sums and also for dependent variables. In the context of random matrices, it was first  extended by Chatterjee \cite{Ch-06} to study matrices with exchangeable entries and was later employed to study different questions in random matrix theory as in   \cite{Tao-Vu-11,Tao-Vu-14,Ba-Me-13, Ba-Me-Pe-13,Ba-14}. In the noncommutative setting, this method was employed to polynomials  in \cite{Kargin} to generalize Voiculescu's {free} CLT and in \cite{Deya-Nourdin} to prove an invariance principle for multilinear homogeneous sums in free elements. In this paper, we extend this method to Cauchy transforms and refine it so that it still provides good quantitative estimates not only in the free case but also in the exchangeable case.

The paper is organized as follows: we recall in Section \ref{Preliminaries} some basic definitions and tools from free probability theory. In Section \ref{section:Lindeberg}, the Lindeberg method is extended to the noncommutative setting, then applied to general sums in free or exchangeable elements, and finally to operator-valued matrices with exchangeable entries.  Being of independent interest, the proof of  the main step of the exchangeable case is given separately in  Section \ref{Section:proof-main-result}. These results are then applied in Section \ref{section:Op-v-matrices} to study the distribution of operator-valued matrices of the form $\frac{1}{\sqrt{2}}(A_N+A_N^*)$ and in Section \ref{section:Wigner} for operator-valued Wigner matrices.  These results are  in turn extended in Section \ref{section:Wishart} to cover operator-valued Wishart matrices. Applications to block random matrices are also given all sections.

\tableofcontents

\section{Preliminaries}\label{Preliminaries}
In this section, we fix the notation and review some concepts and tools from the scalar- and operator-valued free probability theory. 
\paragraph*{\bf Noncommutative probability spaces}
We refer to a pair $(\M,\tau)$, consisting of a von Neumann algebra $\M$ and a faithful normal tracial state $\tau:\M\to \mathbb{C}$, as a \emph{tracial $W^*$-probability space}. 
The \emph{operator norm} is denoted by $\|\cdot\|_{L^\infty(\M,\tau)}$, or by $\|\cdot\|_{\infty}$ when the context is clear. The \emph{$L^p$-norms} are defined, for all $p\geq 1$ and all $x\in \M$, by
$$\|x\|_{L^p(\M,\tau)}=[\tau |x|^p]^{1/p},$$
and can be denoted by $\|\cdot\|_{L^p}$ when the context is sufficiently clear. We  consider random elements in $\M$ as elements of $L^{\infty -}(\Omega,\mathbb{P};\M):= \bigcap_{1\leq p<\infty} L^p(\Omega,\mathbb{P};\M)$ with $\mathbb{E} \circ \tau$ as state. Here, $L^p(\Omega,\mathbb{P})$ is some classical probability space and $\mathbb{E}$ is the associated expectation. 

If $(\M,\tau)$ and $(\N,\varphi)$ are two tracial $W^*$-probability spaces, the tracial $W^*$-probability space $(\M\otimes \N,\tau\otimes \varphi)$ is the tensor product $\M\otimes \N$ of von Neumann algebra endowed with the unique faithful normal tracial state $\tau\otimes \varphi$ such that, for all $x\in \M$ and $y\in \N$, $\tau\otimes \varphi(x\otimes y)=\tau(x)\varphi(y).$
For all $x\in \M$ and $y\in \N$ and $p\in [1,\infty]$ we have
$$\|x\otimes y\|_{L^p(\M\otimes \N,\tau\otimes \varphi)} =\|x\|_{L^p(\M,\tau)}\cdot \|y\|_{L^p(\N,\varphi)}.$$
Finally, for all $x\in \M$ and $y\in \N$, we let $x\otimes^* y := x \otimes y  + x^* \otimes y^*$.

\paragraph*{\bf Exchangeability} Elements $x_1,\ldots,x_n\in \M$ are said to be \emph{exchangeable} with respect to $\tau$ if, for any polynomial $P$ in $2n$ noncommutative  variables, and any permutation $\sigma:\{1,\ldots,n\}\to \{1,\ldots,n\}$, we have
$$\tau[P(x_1,\ldots,x_n,x_1^*,\ldots,x_n^*)]=\tau[P(x_{\sigma(1)},\ldots,x_{\sigma(n)},x_{\sigma(1)}^*,\ldots,x_{\sigma(n)}^*)].$$
An infinite sequence $(x_n)_{n\geq 1} $ in $\M$ is called \emph{exchangeable} with respect to $\tau$ if, for any $n\geq 1$, $x_1,\ldots,x_n\in \M$ are exchangeable. The tail subalgebra $\mathcal{M}_{tail}$ is defined as $\cap_{n\geq 1} \mathcal{B}_n$, where $\mathcal{B}_n$ is the von Neumann subalgebra generated by $(x_k)_{k\geq n} $.

\paragraph*{\bf Cauchy Transform}
For all element $x\in \M$ with negative imaginary part, we denote by $\R_x$ the \emph{resolvent} $\R_x(z)=(z1_{\M}-x)^{-1}$, which is defined on the upper half-plane $\mathbb{C}^+=\{u+iv\in \mathbb{C}:v>0\}$. If $x$ is self-adjoint, the resolvent captures the spectral properties of $x$: for all $z\in \mathbb{C}^+$,
$$\tau[\R_x(z)]=\int_\mathbb{R}\frac{1}{z-t}\diff \mu_x(t),$$
where $\mu_x$ is the \emph{spectral measure} of $x$, i.e. the unique probability measure on $\mathbb{R}$ such that the moments of $x$ are the same as the moments of the probability measure $\mu_x$. In other words, $\tau\circ \R_x$ is the Cauchy transform of the spectral measure of $x$. As an immediate consequence, the pointwise convergence, as $n$ goes to $\infty$, of $\tau[\R_{x_n}(z)]$ to the Cauchy transform of a measure $\nu$ implies the weak convergence of the spectral measure of $x_n$ to $\nu$.

\paragraph*{\bf Operator-valued semicircular}
Let $t>0$. A \emph{semicircular element with variance} $t$ (or with covariance mapping $\eta:\mathbb{C}\ni z\to tz\in \mathbb{C}$) is just a self-adjoint element $s$ whose spectral measure is the semicircular measure
$$\frac{1}{2\pi t}\sqrt{4t-\lambda^2}\cdot 1_{[-2\sqrt{t},2\sqrt{t}]}(\lambda)\diff \lambda.$$

If $\B$ is a von Neumann subalgebra of $\mathcal{M}$, there exists a unique linear map $\tau[\cdot |\B]:\M\to \B$, called the \emph{conditional expectation} onto $\B$, such that for all $x\in \M$ and $y_1,y_2\in \B$ 
$$ \tau[y_1xy_2|\B]=y_1\tau[x|\B]y_2 \quad
 \text{and} \quad \tau[\tau[x|\B]]=\tau[x].$$
Let $\eta:\B\to \B$ be a completely  positive linear mapping. The \emph{operator-valued Cauchy transform of an operator-valued semicircular element} with covariance mapping $\eta$ is an analytic map $G: \mathbb{H}^+(\B) \to \mathbb{H}^-(\B)$ that satisfies,  
\begin{equation}\label{OVSC-CauchyTransform}
1+\eta(G(b)) G(b) = b G(b) \qquad \text{for all $b\in\mathbb{H}^+(\B)$},
\end{equation}
where $$\mathbb{H}^\pm(\B):= \{ b\in\B \mid \exists \varepsilon>0:\ \pm\Im(b) \geq \varepsilon 1 \}$$ 
being the upper and lower half-plane of $\B$, respectively, where we use the notation $\Im(b) := \frac{1}{2i}(b-b^*)$; see \cite[Theorem 4.1.12]{Speicher-98} for a proof of \eqref{OVSC-CauchyTransform}.
We refer to \cite[Chapter 9]{Mingo-Speicher} for more details on operator-valued semicircular elements and on operator-valued Cauchy transforms.

For our purposes, we will need to study the scalar-valued Cauchy transform $g(z)$, $z \in \mathbb{C}^+$, of some operator-valued semicircular elements. Note that $g(z)=\tau(G(z))$ with $G(z)$ being the unique solution of  \eqref{OVSC-CauchyTransform} with asymptotic behavior $G(z) \sim \frac{1}{z}1_\B$ for $z\rightarrow \infty$. In practice, $G(z)$ is the limit of iterates $g_n=\R_{\eta(g_{n-1})}(z)$ for any initial point $g_0$ with negative imaginary part, see \cite{Helton-RFar-Speicher}.

\paragraph*{\bf Freeness and freeness with amalgamation}We say that unital subalgebras $(\mathcal{M}_i)_{i\in I}$ of $\M$ are \emph{free} if whenever we have $m \geq 2$ and   $x_1,\ldots, x_m\in \M$ such that
 $\tau[x_j]=0$  for $j=1,\ldots, m$, 
 $x_j\in \mathcal{M}_{i(j)}$ with $i(j)\in I$
and  $i(1)\neq i(2)$, $i(2)\neq i(3)$ , $\ldots$, $i(m-1)\neq i(m)$,
then 
$$\tau[x_1\cdots x_m]=0.$$
 Elements in $\M$ are called freely independent or free if the $\ast$-algebras generated by them are free. Similarly we define freeness with amalgamation over the operator-valued probability space $\B \subset \M_i$ for all $i \in I$. The subalgebras $(\mathcal{M}_i)_{i\in I}$ of $\M$ are \emph{free with amalgamation over $\B$} (or free with amalgamation with respect to the conditional expectation $\tau[ \cdot |\B]$), if whenever we have $m \geq 2$ and   $x_1,\ldots, x_m\in \M$ such that
 $\tau[x_j|\B]=0$  for $j=1,\ldots, m$, 
 $x_j\in \mathcal{M}_{i(j)}$ with $i(j)\in I$
and  $i(1)\neq i(2)$, $i(2)\neq i(3)$ , $\ldots$, $i(m-1)\neq i(m)$,
then 
$$\tau[x_1\cdots x_m|\B]=0.$$
Elements $\{x_i\}_{i\in I} \in \M$ are called  free with amalgamation over $\B$ if the algebras  $\langle\B,x_i\rangle_{i \in I}$  are also so. 

\section{The Noncommutative Lindeberg method}\label{section:Lindeberg}
Our  main results  rely on the Lindeberg method  which is a  well-known approximation technique in classical probability theory. The aim of this section is to extend this method to the noncommutative setting of a tracial $W^*$-probability space $(\M, \tau)$, and in particular to the setting of free independence and the setting of exchangeability. We end this section by applying the noncommutative Lindeberg method to operator-valued matrices. 

\subsection{General case}The noncommutative Lindeberg method will allow us to compare  Cauchy transforms of elements of the form $\sum_i x_i$ with those of some other elements $\sum_i y_i$ in $\M$. The main idea behind this technique is to write the difference of the resolvents as a telescoping sum,  replace the $x_i$'s by the $y_i$'s, one at a time, and finally give quantitative error bounds. This swapping process is known as the Lindeberg \emph{replacement trick} and allows such approximations under mild moment conditions. 
\begin{theorem}\label{thm:Lindeberg}
Let $(\M,\tau)$ be a tracial $W^*$-probability space. Let $(x_1,\dots,x_n)$ and $(y_1,\dots,y_n)$ be $n$-tuples of self-adjoint elements in $\M$. Setting for any $i \in \{1, \ldots ,n \}$, 
$$
\mathbf z_i = x_1 +\cdots + x_i+y_{i+1}+ \cdots+ y_n
$$
and $$ \mathbf z_i^0 = x_1 +\cdots + x_{i-1}+y_{i+1}+ \cdots + y_n
$$
then, for any   $z\in \mathbb{C}^+$, we have
\[
\big|\tau [\R_{\mathbf z_n}(z)] - \tau [\R_{\mathbf z_0}(z)] \big| \leq \sum_{i=1}^n (|P_i|+|Q_i|+|T_i|), 
\]
where, for any $i \in \{1, \ldots ,n \}$,
\[
P_i= \tau\big[\R_{\mathbf z_i^0}(z)(x_i-y_i)\R_{\mathbf z_i^0}(z) \big] ,
\]
$$
Q_i=  \tau\big[ \R_{\mathbf z_i^0}(z)x_i\R_{\mathbf z_i^0}(z)x_i\R_{\mathbf z_i^0}(z) \big]  - \tau\big[\R_{\mathbf z_i^0}(z)y_i\R_{\mathbf z_i^0}(z)y_i\R_{\mathbf z_i^0}(z)\big] ,
$$
and
$$
T_i=  \tau\Big[\R_{\mathbf z_i}(z) x_i \R_{\mathbf z_i^0}(z)x_i \R_{\mathbf z_i^0}(z) x_i \R_{\mathbf z_i^0}(z)\Big] 
- \tau\Big[\R_{\mathbf z_{i-1}}(z) y_i\R_{\mathbf z_i^0}(z) y_i\R_{\mathbf z_i^0}(z) y_i\R_{\mathbf z_i^0}(z) 
\Big] .
$$
Moreover, if $x_1,\dots,x_n,y_1,\dots,y_n$ are \emph{random} self-adjoint elements in $\M$, we have 
\[
\Big|\mathbb{E}\Big[ \tau[\R_{\mathbf z_n}(z)] - \tau[\R_{\mathbf z_0}(z)]\Big] \Big| \leq \sum_{i=1}^n (|\mathbb{E}[P_i]|+|\mathbb{E}[Q_i]|+|\mathbb{E}[T_i]|).
\]
\end{theorem}
\begin{remark}
The above theorem can be easily extended to more general differentiable functions $f(x_1, \ldots , x_n)$ in $n$ noncommuting variables and is not restricted to Cauchy transforms of partial sums. In that case, the terms of the telescoping sum are developed using Taylor expansions  and  then the estimates are obtained in terms of  noncommutative partial derivatives. However, the restriction we made was essential for our purposes and in particular for dealing with exchangeable variables in Section \ref{section:exch}.
\end{remark}
\begin{proof}[Proof of Theorem \ref{thm:Lindeberg}]
For any $z \in \mathbb{C}^+$, we write the difference of the resolvents as a telescoping sum as follows: 
$$
R_{\mathbf z_n}(z) -R_{\mathbf z_0}(z) 
= \sum^n_{i=1} R_{\mathbf z_i}(z) - R_{\mathbf z_{i-1}}(z)= \sum^n_{i=1} \big(R_{\mathbf z_i}(z)- \R_{\mathbf z_i^0}(z)\big) - \sum^n_{i=1} \big(\R_{\mathbf z_{i-1}}(z) - \R_{\mathbf z_i^0}(z)\big).
$$
To pursue the proof, we shall use the following expansion:
\begin{lemma}\label{lem:Taylor_resolvents}
 Let ${  a}$ and ${  b}$ be  invertible in $\M$, then 
\[
{  a}^{-1} = {  b}^{-1} + {  a}^{-1}\big({  b} - {  a} \big){  b}^{-1},
\]
and more generally, for each $m\in\mathbb{N}_0$, the Taylor formula (with remainder term) holds:
\begin{align*}
{  a}^{-1} = \sum^m_{k=0}  {  b}^{-1}\big[\big({  b} - {  a}\big){  b}^{-1}\big]^k
              +  {  a}^{-1}\big[\big({  b} - {  a}\big){  b}^{-1}\big]^{m+1}.
\end{align*}
\end{lemma}
\noindent
Applying Lemma \ref{lem:Taylor_resolvents} with $m=3$, gives for any  $i \in \{1, \ldots ,n \}$,
\begin{multline*}
\R_{\mathbf z_i}(z)- \R_{\mathbf z_i^0}(z)
 =   \R_{\mathbf z_i^0}(z)  x_i \R_{\mathbf z_i^0}(z)
  +
\R_{\mathbf z_i^0}(z) x_i \R_{\mathbf z_i^0}(z)  x_i  \R_{\mathbf z_i^0}(z)
\\ 
+\R_{\mathbf z_i}(z) x_i  \R_{\mathbf z_i^0}(z) x_i  \R_{\mathbf z_i^0}(z) x_i  \R_{\mathbf z_i^0}(z),
\end{multline*}
and analogously,
\begin{multline*}
\R_{\mathbf z_{i-1}}(z)- \R_{\mathbf z_i^0}(z)
 =   \R_{\mathbf z_i^0}(z)  y_i \R_{\mathbf z_i^0}(z)
  +
\R_{\mathbf z_i^0}(z) y_i \R_{\mathbf z_i^0}(z) y_i \R_{\mathbf z_i^0}(z)
\\ 
+\R_{\mathbf z_{i-1}}(z) y_i\R_{\mathbf z_i^0}(z)  y_i  \R_{\mathbf z_i^0}(z) y_i \R_{\mathbf z_i^0}(z).
\end{multline*}
The assertion follows by taking their difference and applying $\tau $ or $\mathbb{E}\circ \tau $ in the random variant.
\end{proof}

\subsection{Freeness}\label{section:freeness}
We address in this section  sums of freely independent elements in  $(\M, \tau)$. We shall see how a simple application of the noncommutative Lindeberg method in Theorem \ref{thm:Lindeberg} gives quantitative estimates to the difference of Cauchy transforms under the  mere condition of matching conditional expectations of  first and second orders. 
\begin{theorem}\label{theo:Lindeberg-free}
Let $(\M,\tau)$ be a tracial $W^*$-probability space and $\mathcal B$ be a von Neumann subalgebra of $\M$. Let $(x_1,\dots,x_n)$ and $(y_1,\dots,y_n)$  be two $n$-tuples of self-adjoint elements in $\M$ which are free with amalgamation over $\mathcal B$ and such that $\tau[x_i|\mathcal B]=\tau[y_i|\mathcal B]=0$ and $\tau[x_ibx_i|\mathcal B]=\tau[y_iby_i|\mathcal B]$  for any $b\in \mathcal B$ and $1\leq i \leq n$.
Setting
$$\mathbf x = \sum_{i=1}^nx_i\ \text{and}\ \ \mathbf y= \sum_{i=1}^ny_i$$
then, for any   $z\in \mathbb{C}^+$, we have
\[
\big|\tau [\R_{\mathbf x}(z)] - \tau [\R_{\mathbf y}(z)] \big|
\leq \frac{K_\infty K_2^2}{Im (z)^4 } n, 
\]
where $K_\infty=\max_i \big(\|x_i \|_{\infty}+\|y_i \|_{\infty}\big)$ and $K_2=\max_i (\|x_i\|_{L^2} + \|y_i\|_{L^2})$.
\end{theorem}

\begin{proof}
Let us assume without loss of generality that $(x_1,\dots,x_n,y_1,\dots,y_n)$ are free. The proof is a  direct application of the Lindeberg method in Theorem \ref{thm:Lindeberg}. Note that for any $i = 1 , \ldots , n$, $x_i$ and $y_i$ are centered and free from $\uZ_i^0$ with amalgamation over $\mathcal B$ which implies that
\begin{align*}
\tau \big[\R_{\uZ_i^0}(z) x_i \R_{\uZ_i^0}(z) \big] 
=  \tau\big[\R_{\uZ_i^0}(z) y_i \R_{\uZ_i^0}(z) \big] =0, 
\end{align*}
and thus $P_i=0$. For the same reasons
\begin{align*}
\tau \big[ \R_{\uZ_i^0}(z) x_i \R_{\uZ_i^0}(z) x_i \R_{\uZ_i^0}(z) \big]
&= \tau \big[ \R_{\uZ_i^0}(z)\tau \big[ x_i\R_{\uZ_i^0}(z) x_i|\mathcal B\big]\R_{\uZ_i^0}(z)\big]
\\ &= \tau \big[ \R_{\uZ_i^0}(z)\tau \big[ y_i\R_{\uZ_i^0}(z) y_i|\mathcal B\big]\R_{\uZ_i^0}(z)\big]
\\&= \tau \big[ \R_{\uZ_i^0}(z) y_i \R_{\uZ_i^0}(z) y_i \R_{\uZ_i^0}(z) \big],
\end{align*}
implying that $Q_i = 0$. Considering the third order term $T_i$, we compute
\begin{align*}
\tau\Big[\R_{\mathbf z_i}(z) \big(x_i  \R_{\mathbf z_i^0}(z) \big)^3
\Big] 
&\leq  \|\R_{\mathbf z_i}(z) \big(x_i \R_{\mathbf z_i^0}(z) \big)^3\|_{L^1}\\
& \leq  \| \R_{\mathbf z_i}(z)\|_\infty \| \R_{\mathbf z_i^0}(z)\|_\infty^3 \| x_i \|_\infty  \| x_i \|_{L^2}^2
\\& \leq \frac{1}{Im (z) ^4}\|x_i \|_{\infty} \| x_i \|_{L^2}^2,
\end{align*}
and similarly, we control the second term in $T_i$ and get that $T_i\leq \frac{K_\infty K_2^2}{Im (z)^4 }$. We finally conclude by summing over $i$.
\end{proof}

\subsection{Exchangeable variables}\label{section:exch}
In this section, we  drop the free independence hypothesis and treat instead exchangeable variables in $(\M, \tau)$.  
We shall see how we can approximate  sums of exchangeable operators with independent Gaussian operators 
and give again quantitative estimates to such approximations in terms of Cauchy transforms.

\begin{theorem}\label{theo:Lindeberg-exch}
Let $(\M,\tau)$ and $(\N,\varphi)$ be two tracial $W^*$-probability spaces. Let $(x_1,\dots,x_n)$ be an $n$-tuple of \emph{exchangeable} elements in $\M$.  Consider a family $(N_{i,k})_{1\leq i,k \leq n}$  of independent standard Gaussian random variables and let $(y_1,\dots,y_n)$ be the $n$-tuple of random elements in $\M$ given by
$$y_i=\frac{1}{\sqrt{n}}\sum_{k=1}^n N_{i,k}\left( x_k-\frac{1}{n}\sum_{j=1}^nx_j\right).$$
Let $(a_1,\dots,a_n)$ be an $n$-tuple of elements in $\N$ and set
$$\mathbf x = \sum_{i=1}^nx_i\otimes a_i+x_i^*\otimes a_i^*\ \ \text{and}\ \ \mathbf y= \sum_{i=1}^ny_i\otimes a_i+y_i^*\otimes a_i^*.$$
Then, for any $z\in \mathbb{C}^+$,
$$\big|\tau\otimes \varphi(\R_{\mathbf x}(z)) - \mathbb{E}[\tau\otimes\varphi(\R_{\mathbf y}(z))] \big| \leq \frac{K}{(\Im(z)\wedge 1)^5}n^{-1/4},$$
with
$$
K=C\cdot\left( K_1\|x_1\|_{\infty}n^{1/4}+ K_2^2\|x_1\|^2_{\infty}n^{3/4} +K_\infty K_2^2\|x_1\|_{\infty}^3n^{5/4}+K_\infty^2K_2^2\|x_1\|^4_{\infty}n^{5/4}\right),$$
where $K_\infty=\max_i \| a_i\|_{\infty}, K_2=\max_i \| a_i\|_{L^2}$, $K_1= \|\sum_{i=1}^n a_i\|_{L^1}$ and $C$ is a universal constant. 
\end{theorem}

In the applications given in the next sections, we will compute the values of $K_\infty$, $K_2$ and $K_1$, which would yield  the boundedness of $K$.

Approximating sums of exchangeable elements with semicirculars using our approach was not possible even when imposing stronger conditions, like quantum exchangeability \cite[page 108]{book-Voi-Sta-Web}. This intermediate approximation with random operators was crucial to our purposes.

Being long and technical, the proof of  Theorem \ref{theo:Lindeberg-exch} is postponed to Section \ref{Section:proof-main-result}.  The results we have presented so far cover quite general sums in tracial $W^*$-probability spaces. In the following, we specify our object of interest and study precisely operator-valued matrices. 

\subsection{Application to an operator-valued matrix}\label{Section:Op-v-exch-entries} In this section, we shall apply the Lindeberg method to the study  of matrices with operator-valued  entries which are exchangeable. We consider a \emph{finite} family of exchangeable operators as entries and prove that  the associated matrix is close in distribution to an operator-valued semicircular element by providing explicit rates of convergence for the Cauchy transforms. 

Let  $(\M, \tau)$ be a tracial $W^*$-probability space and  $\uX=(x_{ij})_{1\leq i,j\leq N} $ be a finite  family of elements in $\M$. We consider the $N \times N$ operator-valued  matrix $A(\ux)$  given by 
\[
[A(\ux)]_{ij}=\frac{1}{\sqrt{2N}}(x_{ij}+x_{ji}^*).
\]
To study its behavior, we  consider its Cauchy transform and apply the results in Sections \ref{section:freeness} and \ref{section:exch}.  With this aim, we set $n=N^2$ and note that the associated linear map {$A : \M^n \rightarrow M_N(\M)$} can be written as follows:
\begin{equation}\label{def:map_matrix}
A(\ux) = \sum_{ i,j=1}^N \big(x_{ij} \otimes a_{ij}+ x_{ij}^* \otimes a_{ij}^* \big)=:\sum_{i,j=1 }^N x_{ij} \otimes^* a_{ij},
\end{equation}
where $a_{ij}=  \frac{1}{ \sqrt{2N}} E_{ij}$ and $(E_{ij})_{1\leq i,j \leq N}$ are the canonical $N\times N$ matrices. 

\begin{theorem}\label{thm:Op-v-matrices-exch}
Let  $(\M, \tau)$ be a tracial $W^*$-probability space with $\M$ being  a direct integral of finite dimensional factors of dimension $\leq d$. Let  $\ux=(x_{ij})_{1\leq i,j\leq N} $ be a family of exchangeable variables in $\M^n$ with $n=N^2$. Set $\bar{x}_{ij}=x_{ij}-\mu$ where $\mu=\frac{1}{n}\sum_{1\leq i,j\leq N} x_{ij}$ and consider the variance function 
\[
\eta : \M \rightarrow \M, 
\qquad 
\eta(b)=\frac{1}{2n}\sum_{i,j=1}^N \Big(\bar{x}_{ij}b\bar{x}_{ij}^*+\bar{x}_{ij}^*b\bar{x}_{ij}\Big).
\] 
 Let $\mathcal{B}$ be the smallest $W^*$-algebra which is closed under $\eta$ and consider a centered $\B$-valued semicircular variable $S$ of variance $E_{\mathcal{B}}(SbS)= \eta(b)$.
 Letting $A(\uX) \in \M_N(\M)$ be the operator-valued matrix defined in \eqref{def:map_matrix} 
then 
there exists a universal constant $C>0$ such that, for any $z\in \mathbb{C}^+$,
$$
\big|(\tau \otimes \tr_N)[\R_{A(\uX)}(z)] - \tau[\R_{S}(z)]\big|
\leq C d^{2}\frac{(|z|\vee \|x_{11}\|_\infty\vee 1)^4}{(\Im(z)\wedge 1)^6}\frac{1}{\sqrt{N}}.
$$
\end{theorem}

As our estimates hold over all the upper complex half-plane, then once the cumulative distribution function of the operator-valued semicircular element $S$ is \emph{H\"older continuous}, the estimates on the Cauchy transforms can be passed immediately to quantitative estimates on the Kolmogorov distance by a simple application of the following theorem from~\cite{Banna-Mai-20}.  This will be the case for some examples we present later in the paper.
\begin{theorem}[Theorem 5.2 of \cite{Banna-Mai-20}]\label{BannaMai}Let $\mu_n$ be a sequence of Borel probability measures on $\mathbb{R}$ and let $\nu$ be any other Borel probability measure on $\mathbb{R}$. Suppose the following:
\begin{itemize}
    \item The cumulative distribution function of $\nu$ is Hölder continuous with exponent $\beta\in (0,1]$.
    \item There are $K>0$, $\rho>0$, $\alpha>0$, $l\geq 0$ and $k\geq 0$ such that, for all $z\in \mathbb{C}^+$ with $\Im(z)<\rho$, the difference between the Cauchy transforms $G_{\mu_n}$ and $G_\nu$ satisfies
    $$\left|G_{\mu_n}(z)-G_{\nu}(z)\right|\leq K\frac{|z|^l}{\Im(z)^k}N^{-\alpha}.$$
    \item We have that $\sup_n \int_\mathbb{R} t^2 d\mu_n(t)<\infty$.
\end{itemize}
Then, $(\mu_n)_{n=1}^\infty$ converges in Kolmogorov distance to $\nu$; in fact, there is $c>0$ such that
$$ {\rm Kol} (\mu_n,\nu)\leq c \cdot N^{-\alpha\cdot \frac{\beta}{2+k+(2-\beta)l}}.$$
\end{theorem}

The proof of Theorem \ref{thm:Op-v-matrices-exch} is divided into  two parts: first we shall apply the Lindeberg method  to approximate the Cauchy transform of $A(\uX)$ by that of a random Gaussian operator-valued matrix, and then prove that the later is close to the Cauchy transform of the semicircular element $S$. 

The first step will be accomplished in the following proposition  which is an application of the Lindeberg method in the exchangeable setting proved in Section \ref{section:exch}:

\begin{proposition}\label{thm:RMstieltjesone}
Let  $(\M, \tau)$ be a tracial $W^*$-probability space. Let  $\ux=(x_{ij})_{1\leq i,j\leq N} $ be a family of exchangeable variables in $\M^n$ with $n=N^2$. Set $\bar{x}_{ij}=x_{ij}-\mu$ where $\mu=\frac{1}{n}\sum_{1\leq i,j\leq N} x_{ij}$. 
Consider a family $(N_{ij,kl})_{1\leq i,j \leq N,1\leq k,l \leq N}$  of $n^2$ independent Gaussian random variables. We consider the random operators $\uY=(y_{ij})_{1\leq i,j \leq N}  $ given by
\[
y_{ij}=\frac{1}{\sqrt{n}}\sum_{k,l=1}^N N_{ij,kl}\left(x_{kl}-\mu\right).
\]
Then there exists a universal constant $C>0$ such that, for any $z\in \mathbb{C}^+$,
$$
\big|(\tau \otimes \tr_N)[\R_{A(\uX)}(z)] - \mathbb{E}\big[(\tau \otimes \tr_N)[\R_{A(\uY)}(z)]\big] \big|
 \leq C\frac{(\|x_1\|_{\infty}\vee 1)^4}{(\Im(z)\wedge 1)^5} \frac{1}{\sqrt{N}} 
.
$$
\end{proposition}

\begin{proof}
This is a direct consequence of the general result in Theorem \ref{theo:Lindeberg-exch} with $(\N,\varphi)=(M_N(\mathbb{C}), \tr_N)$ and $(a_1, \ldots, a_n)= (\frac{1}{\sqrt{N}} E_{ij}, 1\leq i,j \leq N)$ is the tuple of the $N \times N$ normalized canonical matrices. The result follows by noting that in this case $K_\infty=\mathcal{O}(1/ \sqrt{N})$, $K_2=\mathcal{O}(1/N)$ and $K_1=\mathcal{O}(1/ \sqrt{N})$ which guarantee that $K=\mathcal{O}(1)$.
\end{proof}

We have reached now the last step in our proof in which we approximate the Gaussian operator-valued matrix $A(\uY)$ by an operator-valued semicircular element.
\begin{proposition}\label{thm:RMstieltjestwo}Let  $(\M, \tau)$ be a tracial $W^*$-probability space with $\M$ being  a direct integral of finite dimensional factors of dimension $\leq d$. Let  $\uX=(x_{ij})_{1\leq i,j\leq N} $ be a family of  variables in $\M^n$, and let us denote by $\mu$ their mean.
Consider a family $(N_{ij,kl})_{1\leq i,j \leq N,1\leq k,l \leq N}$  of $n^2$ independent Gaussian random variables. We consider the random operators $\mathbf{y}=(y_{ij})_{1\leq i,j\leq N}$ of the previous theorem and the variance function \[
\eta : \M \rightarrow \M, 
\qquad 
\eta(b)=\frac{1}{2n}\sum_{i,j=1}^N \Big(\bar{x}_{ij}b\bar{x}_{ij}^*+\bar{x}_{ij}^*b\bar{x}_{ij}\Big).
\]  Let $\mathcal{B}$ be the smallest $W^*$-algebra which is closed under $\eta$ and consider a centered semicircular variable $S$ over  $\mathcal{B}$ of variance $E_{\mathcal{B}}(SbS)= \eta(b)$. Letting $A$ be the linear map defined in \eqref{def:map_matrix} then there exists a universal constant $C>0$ such that for any $z\in \mathbb{C}^+$,
$$\left\|\mathbb{E}[\tr_N(\R_{A(\mathbf{y})}(z))]-E_{\B}[\R_{S}(z)]\right\|_{L^2(\M , \tau)}
\leq C d^{2}\frac{(\|x\|_\infty \vee |z|\vee  1)^3}{(\Im(z)\wedge 1)^6}\frac{1}{\sqrt{N}}$$
where $\|x\|_\infty:=\max_{1\leq j \leq i \leq N} \|x_{ij}\|_\infty$.
\end{proposition}
\begin{remark}
Exchangeability doesn't play any role in the above proposition which is valid for any family $\uX$ in $\M^n$. The independent Gaussian random variables and the structure of the Wigner matrix  are what allow such an approximation.
\end{remark}

\begin{proof}Because $\M$ is a direct integral of finite dimensional factors of dimension $\leq d$, the general case follows from the case $\M=M_{d'}(\mathbb{C})$ with $d'\leq d$. Indeed, we can reason almost everywhere with respect to this integral decomposition, and deduce the general inequality from the matrix case. As a consequence, without loss of generality, we assume in the following that $\M=M_{d'}(\mathbb{C})$ with $d'\leq d$. 

Let us denote by $s$ the map $z\mapsto E_{\mathcal{B}} (\R_{S}(z))$ and by $g$ the map $z\mapsto \mathbb{E}[\tr_N(\R_{A(\uY)}(z))]$. Remark that
$$0=\eta(s(z))s(z)-zs(z)+1.$$
We set
$$\delta:=\eta(g(z))g(z)-zg(z)+1 $$
\paragraph{\bf{Step 1}}We begin by showing that $\|\delta\|_{L^2(\M , \tau)}=\mathcal{O}_z(\frac{1}{N})$.

With this aim, we start by noticing that
$$1=(z-A(\uY))\R_{A(\uY)}(z)=z\R_{A(\uY)}(z)-A(\uY)\R_{A(\uY)}(z).$$
Note that
$$
\E[A(\uY)\R_{A(\uY)}(z)]=\frac{1}{\sqrt{2nN}}\sum_{i,j,k,l=1}^N\bar{x}_{kl}\E[N_{ij,kl}E_{ij} \R_{A(\uY)}(z)]+\bar{x}_{kl}^*\E[N_{ij,kl}E_{ji} \R_{A(\uY)}(z)],$$
and then integration by parts with respect to $N_{ij,kl}$ gives the following terms
\begin{multline*}
 \bar{x}_{kl}\E[E_{ij} \R_{A(\uY)}(z)E_{ij}\bar{x}_{kl}\R_{A(\uY)}(z)]+\bar{x}_{kl}\E[E_{ij} \R_{A(\uY)}(z)E_{ji}\bar{x}_{kl}^*\R_{A(\uY)}(z)]\\
+\bar{x}_{kl}^*\E[E_{ji} \R_{A(\uY)}(z)E_{ij}\bar{x}_{kl}\R_{A(\uY)}(z)] +\bar{x}_{kl}^*\E[E_{ji} \R_{A(\uY)}(z)E_{ji}\bar{x}_{kl}^*\R_{A(\uY)}(z)],
\end{multline*}
and thus $ \E[A(\uY)\R_{A(\uY)}(z)]$ can be rewritten as
\begin{multline*}
E[\eta(\tr_N(\R_{A(\uY)}(z)))\R_{A(\uY)}(z)] 
+\frac{1}{2nN}\sum_{k,l=1}^N\bar{x}_{kl}\E[\Theta(\R_{A(\uY)}(z))\bar{x}_{kl}\R_{A(\uY)}(z)] \\
 +\frac{1}{2nN}\sum_{k,l=1}^N\bar{x}_{kl}^*\E[\Theta( \R_{A(\uY)}(z))\bar{x}_{kl}^*\R_{A(\uY)}(z)],
\end{multline*}
where $\Theta:M_N(\M)\to M_N(\M)$ is the partial transpose $\Theta(M \otimes x)=M^t\otimes x$. Note that by Lemma \ref{Normtranspose}
\begin{align*}
\Big\|\sum_{k,l=1}^N\bar{x}_{kl}\Theta(\R_{A(\uY)}(z))\bar{x}_{kl}\Big\|_{L^2}
&\leq  4 n\|x\|_\infty^2\|\Theta(\R_{A(\uY)}(z))\|_{L^2}
\\& = 4 n\|x\|_\infty^2\|\R_{A(\uY)}(z)\|_{L^2}\leq 4 \frac{n\|x\|_\infty^2}{\Im(z)},
\end{align*}
where we recall that $\|x\|_\infty:=\max_{1\leq j \leq i \leq N} \|x_{ij}\|_\infty$. Finally, applying $\mathbb{E}$ and $\tr_N$, $1=z\R_{A(\uY)}(z)-A(\uY)\R_{A(\uY)}(z)$ becomes
$$1=zg(z)-\mathbb{E}[\eta(\tr_N(\R_{A(\uY)}(z)))\tr_N(\R_{A(\uY)}(z))]+ \mathcal{O}_z\Big(\frac{1}{N }\Big).\label{eq:res}$$
Now, we shall show that we can replace
$\mathbb{E}[\eta(\tr_N(\R_{A(\uY)}(z)))\tr_N(\R_{A(\uY)}(z))]$ by $\eta(g(z))g(z)$ with an error of order $\mathcal{O}_z(\frac{1}{N})$, which allows to conclude that $\|\delta\|_{L^2(\M , \tau)}=\mathcal{O}_z(\frac{1}{N})$. We write
$$
\mathbb{E}[\eta(\tr_N(\R_{A(\uY)}(z)))\tr_N(\R_{A(\uY)}(z))]-\eta(g(z))g(z)=\mathbb{E}[\eta(\tr_N(\R_{A(\uY)}(z)))\Delta],
$$
where we denote by $\Delta$ the centered variable 
$
\tr_N(\R_{A(\uY)}(z))-\mathbb{E}[\tr_N(\R_{A(\uY)}(z))].
$
We have
\begin{align*}
\left\|\mathbb{E}[\eta(\tr_N(\R_{A(\uY)}(z)))\Delta]\right\|_{L^2}
&\leq\E\left\|\eta(\tr_N(\R_{A(\uY)}(z)))\Delta\right\|_{L^2}\\
&\leq \frac{4\|x\|_\infty^2}{\Im(z)}\E\|\Delta\|_{L^2}
\\&\leq \frac{4\|x\|_\infty^2}{\Im(z)^2}(\E\|\Delta\|_{L^2}^2)^{1/2},
\end{align*}
where the last term can be bounded as follows
\begin{equation}\label{norme2-delta}
\E \|\Delta\|_{L^2}^2 =\E\left\|\tr_N(\R_{A(\uY)}(z))-\E[\tr_N(\R_{A(\uY)}(z))]\right\|_{L^2}^2
\leq\frac{ 2 \pi^2d^2\|x\|_\infty^2}{ n \, \Im(z)^4}.
\end{equation}
To prove  \eqref{norme2-delta} we shall apply a Gaussian concentration result of Pisier, which we state here:
\begin{theorem}[Theorem 2.2, \cite{Pisier}] \label{theo:LedouxOl}
Let $F$ be a normed vector space and $\Phi : F \rightarrow \mathbb{R}$ be a convex function. Letting $\mathbf{N}=(N_1, \ldots , N_n)$ be a standard Gaussian random vector in $\mathbb{R}^n$ then for any smooth and sufficiently integrable function $f: \mathbb{R}^n \rightarrow F$ such that $\E [f(\mathbf{N})]=0$, we have
\[
\E \big[ \Phi \big( f(\mathbf{N}) \big) \big] \leq \E \Big[ \Phi \Big( \frac{\pi}{2} \mathbf{N'} \cdot \nabla f(\mathbf{N}) \Big) \Big] ,
\]
where $ \mathbf{N'} $ is an independent copy of $ \mathbf{N}$.  
\end{theorem}
Let $B: \mathbb{R}^{n^2} \rightarrow M_N(\M)$ be the map defined by
\[
B \big( (w_{ij,kl})_{1 \leq i,j \leq N , 1 \leq k,l \leq N}\big) = \frac{1}{\sqrt{2nN}}\sum_{i,j,k,l=1}^N w_{ij,kl}\bar{x}_{kl}\otimes^* E_{ij}. 
\] 
 Letting $\mathbf{N}= (N_{ij,kl})_{j\leq i, l \leq k}$, we note   that $B(\mathbf{N}) = A(\uY)$. Then applying Theorem \ref{theo:LedouxOl} for  $f= \tr_N(\R_{B (.)}(z)) - \E[\tr_N(\R_{B(.)}(z))] $ and $\Phi = \| . \|^2_{L^2(\M, \tau)}$ gives
$$
\mathbb{E}\left\|\tr_N(\R_{A(\uY)}(z))-\mathbb{E}[\tr_N(\R_{A(\uY)}(z))]\right\|_{L^2}^2
\leq \frac{\pi^2}{4} \mathbb{E}\left\|\tr_N\big(\R_{A(\uY)}(z) Y' \R_{A(\uY)}(z)\big)\right\|_{L^2}^2,
$$
where  $Y'= B(\mathbf{N'})$ with $\mathbf{N'}$ being an independent copy of $\mathbf{N}$.
Now to control the right-hand term, we proceed by conditioning on $\uY$:
\begin{align*} 
\!&\mathbb{E}\big[\left\|\tr_N(\R_{A(\uY)}(z) Y' \R_{A(\uY)}(z))\right\|_{L^2}^2|\mathbf{N} \big]
\\&=\tau \big( \mathbb{E} \big[\tr_N(\R_{A(\uY)}(z) Y'\R_{A(\uY)}(z))\tr_N(\R_{A(\uY)}(z)^*Y'\R_{A(\uY)}(z)^*)|\mathbf{N} \big] \big)\\
&=\frac{1}{2Nn} \sum_{i,j,k,l=1}^N\tau \big(\tr_N \big(\R_{A(\uY)}(z)(\bar{x}_{kl} \otimes^* E_{ij} )\R_{A(\uY)}(z))\tr_N(\R_{A(\uY)}(z)^*(\bar{x}_{kl} \otimes^* E_{ij})\R_{A(\uY)}(z)^*\big)\big)
\\
&=\frac{1}{2n^2} \sum_{k,l} \Big[\tau\big(\tr_N \big(\Theta(\R_{A(\uY)}(z))(I_N \otimes \bar{x}_{kl}^*)\Theta(\R_{A(\uY)}(z))\Theta(\R_{A(\uY)}(z)^*)(I_N\otimes \bar{x}_{kl})\Theta(\R_{A(\uY)}(z)^*)\big)\big)\\
&\quad+ 
\tau\big(\tr_N \big(\Theta(\R_{A(\uY)}(z))(I_N \otimes \bar{x}_{kl})\Theta(\R_{A(\uY)}(z))\Theta(\R_{A(\uY)}(z)^*)(I_N\otimes \bar{x}_{kl}^*)\Theta(\R_{A(\uY)}(z)^*)\big)\big)\\
&\quad+
\tau\Big(\tr_N \Big(\Theta\big(\Theta(\R_{A(\uY)}(z))(I_N \otimes \bar{x}_{kl}^*)\Theta(\R_{A(\uY)}(z))\big)\Theta(\R_{A(\uY)}(z)^*)(I_N\otimes \bar{x}_{kl}^*)\Theta(\R_{A(\uY)}(z)^*)\Big)\Big)\\
&\quad +
\tau\Big(\tr_N \Big(\Theta\big(\Theta(\R_{A(\uY)}(z))(I_N \otimes \bar{x}_{kl})\Theta(\R_{A(\uY)}(z))\big)\Theta(\R_{A(\uY)}(z)^*)(I_N\otimes \bar{x}_{kl})\Theta(\R_{A(\uY)}(z)^*)\Big)\Big) \Big],
\end{align*} 
where $\Theta:M_N(\M)\to M_N(\M)$ is the partial transpose $\Theta(M \otimes x)=M^t\otimes x$.
As a consequence, using Lemma~\ref{Normtranspose},
\begin{equation*}
\mathbb{E}\big[\left\|\tr_N(\R_{A(\uY)}(z) A(\uY)' \R_{A(\uY)}(z))\right\|_{L^2}^2|\mathbf{N} \big]
\leq \frac{8d^2}{n}\frac{\|x\|_\infty^2}{\Im(z)^4}
\end{equation*}
and finally, we get \eqref{norme2-delta}\begin{equation*}
\E \|\Delta\|_{L^2}^2 =\E\left\|\tr_N(\R_{A(\uY)}(z))-\E[\tr_N(\R_{A(\uY)}(z))]\right\|_{L^2}^2
\leq\frac{ 2 \pi^2d^2\|x\|_\infty^2}{ n \, \Im(z)^4}.
\end{equation*}
Equation \eqref{norme2-delta} allows to write the bound
\begin{align*}
\left\|\mathbb{E}[\eta(\tr_N(\R_{A(\uY)}(z)))\Delta]\right\|_{L^2}
&\leq\frac{4\sqrt{2}\pi d\|x\|_\infty^3}{\Im(z)^4}\frac{1}{N}
\end{align*}
which was the last step to prove that $\|\delta\|_{L^2(\M , \tau)}=\mathcal{O}_z(\frac{1}{N})$. More precisely, we have just proved that
\begin{equation}
    \|\delta\|_{L^2(\M , \tau)}\leq \frac{4\|x\|_\infty^2}{\Im(z)}\frac{1}{N} +\frac{4\sqrt{2}\pi d\|x\|_\infty^3}{\Im(z)^4}\frac{1}{N}\leq 4\sqrt{2}\pi d \frac{(\|x\|_\infty\vee 1)^3}{(\Im(z)\wedge 1)^5}\frac{1}{N}.\label{eq:step.one}
\end{equation}
\paragraph{\bf{Step 2}} The second step consists in linking $s$ and $g$, following the approach of Haagerup and Thorbjorsen \cite{Haagerup-Thorbjorsen}. First, we note that following the proof of Proposition 5.2 of\cite{Haagerup-Thorbjorsen}, we know that $g(z)$ is invertible and that
$$\left\|\frac{1}{g(z)}\right\|_\infty\leq \frac{(|z|+\mathbb{E}[\|A(\uY)\|_\infty])^2}{\Im(z)}.$$
As a consequence, we can set $\Lambda(z):=z+\delta g(z)^{-1}$  in such a way that $g(z)$ is a solution to the equation
 \begin{equation}
     \Lambda(z)\R-\eta(\R)\R=1.\label{eq:fixedpoint}
 \end{equation}
 Whenever $\Im(z)>2\|\Lambda(z)-z\|_{\infty}$, we have
 $$\Im(\Lambda(z))=\Im(z)+\Im(\Lambda(z)-z)\geq \Im(z)-\|\Lambda(z)-z\|_{\infty}> \Im(z)/2>0,$$
 and it was shown in \cite{Helton-RFar-Speicher} that, in this case, there exists exactly one solution to the fixed point equation \eqref{eq:fixedpoint} whose imaginary part is negative. Noting that
 $$\Lambda(z)s(\Lambda(z))-\eta(s(\Lambda(z)))s(\Lambda(z))=1,$$
 the uniqueness of the solution yields that $ g(z)=s(\Lambda(z))$. This allows us to write
 \begin{align*}
     \|g(z)-s(z)\|_{L^2(\M , \tau)}&=\|s(\Lambda(z))-s(z)\|_{L^2(\M , \tau)}\\
     &=\|s(\Lambda(z))(\Lambda(z)-z)s(z)\|_{L^2(\M , \tau)}\\
     &\leq \left\|\frac{1}{\Im(\Lambda(z))}\right\|_\infty\|\Lambda(z)-z\|_{L^2(\M , \tau)}\left\|\frac{1}{\Im(z)}\right\|_\infty\\
     &\leq 2\frac{\|\Lambda(z)-z\|_{L^2(\M , \tau)}}{\Im(z)^2}.
 \end{align*}
Whenever $\Im(z)\leq 2 \|\Lambda(z)-z\|_\infty$, we have
    $$\|g(z)-s(z)\|_{L^2(\M , \tau)}\leq \frac{2}{\Im(z)}\leq 4\frac{\|\Lambda(z)-z\|_\infty}{\Im(z)^2}.$$
In any case, for all $z\in \mathbb{C}^+$, 
$$\|g(z)-s(z)\|_{L^2(\M , \tau)}\leq 4\frac{\|\Lambda(z)-z\|_\infty}{\Im(z)^2}\leq 4 \|\delta\|_{\infty} \left\|\frac{1}{g(z)}\right\|_\infty.$$
Here, as $\delta\in M_{d'}(\mathbb{C})$, we have $\|\delta\|_{\infty}\leq \sqrt{d'}\|\delta\|_{L^2(\M , \tau)}\leq \sqrt{d}\|\delta\|_{L^2(\M , \tau)}$ and we get
\begin{equation}\|g(z)-s(z)\|_{L^2(\M , \tau)}\leq 4\sqrt{d} \|\delta\|_{L^2(\M , \tau)} \left\|\frac{1}{g(z)}\right\|_\infty.\label{eq:step.two}\end{equation}
\paragraph{\bf{Step 3}}This step consists in bounding $\left\|\frac{1}{g(z)}\right\|_\infty$.
As mentioned above, we have
$$\left\|\frac{1}{g(z)}\right\|_\infty\leq \frac{(|z|+\mathbb{E}[\|A(\uY)\|_\infty])^2}{\Im(z)}.$$ 
Recall that $A(\uY)\in M_N(\mathcal{M})$. Now, as $(\M, \tau)=(M_{d'}(\mathbb{C}), \tr_{d'})$ then $A(\uY)$ is in fact a matrix of size $Nd'\times Nd'$ whose operator norm can be bounded using the inequality
    $$\|A(\uY)\|_\infty\leq \left(\Tr(A(\uY)^4)\right)^{1/4}= (Nd')^{1/4}\left(\tau\left(\tr_N(A(\uY)^4)\right)\right)^{1/4}.$$
As a consequence, and as $d'\leq d$, we get
     $$\left\|\frac{1}{g(z)}\right\|_\infty\leq \frac{(|z|+\mathbb{E}[\|A(\uY)\|_\infty])^2}{\Im(z)}\leq \frac{(|z|+(Nd)^{1/4}\mathbb{E}[\tau\left(\tr_N(A(\uY)^4)\right)]^{1/4})^2}{\Im(z)}.$$
     Let us prove that $\mathbb{E}[\tau\left(\tr_N(A(\uY)^4)\right)]=O(1)$. We recall that
 \begin{align*}
    A(\uY)&=\frac{1}{\sqrt{2N}}\sum_{i,j=1}^N\left(\frac{1}{\sqrt{n}}\sum_{k,l=1}^NN_{ij,kl}\bar{x}_{kl}^*+N_{ji,kl}\bar{x}_{kl}^*\right)\otimes E_{ij}\\
    &=\frac{1}{\sqrt{2Nn}}\sum_{i,j,k,l=1}^N \left(N_{ij,kl}\bar{x}_{kl}^*+N_{ji,kl}\bar{x}_{kl}^*\right)\otimes E_{ij} ,
\end{align*} and we compute
\begin{align*}
    \tr_N(A(\uY)^4)&=\frac{1}{4n^3}\sum_{\substack{i_1,i_2,i_3,i_4\\k_1,k_2,k_3,k_4 \\ l_1,l_2,l_3,l_4=1}}^N \left(N_{i_1i_2,k_1l_1}\bar{x}_{k_1l_1}^*+N_{i_2i_1,k_1l_1}\bar{x}_{k_1l_1}^*\right)\cdots \left(N_{i_4i_1,k_4l_4}\bar{x}_{k_4l_4}^*+N_{i_1i_4,k_4l_4}\bar{x}_{k_4l_4}^*\right).
\end{align*}
In order to have an upper bound on the quantity above, remark that, whenever $N_1,\ldots,N_4$ are Gaussian variables positively correlated, and $x_1,\ldots,x_4\in \M$, we have
\begin{align*}
   \left| \mathbb{E}\left[\tau\left(N_1x_1N_2x_2N_3x_3N_4x_4\right)\right]\right|&=\left|\mathbb{E}\left[N_1N_2N_3N_4\right]\tau\left(x_1x_2x_3x_4\right)\right|\\
   &\leq \mathbb{E}\left[N_1N_2N_3N_4\right] \|x_1\|_\infty \|x_2\|_\infty \|x_3\|_\infty \|x_4\|_\infty.
\end{align*}
As a consequence, for fixed indices,
\begin{align*}&\left|\mathbb{E}\left[\tau\left(\left(N_{i_1i_2,k_1l_1}\bar{x}_{k_1l_1}^*+N_{i_2i_1,k_1l_1}\bar{x}_{k_1l_1}^*\right)\cdots \left(N_{i_4i_1,k_4l_4}\bar{x}_{k_4l_4}^*+N_{i_1i_4,k_4l_4}\bar{x}_{k_4l_4}^*\right)\right)\right]\right|\\
 & \leq\mathbb{E}\left[\left(N_{i_1i_2,k_1l_1}\|\bar{x}_{k_1l_1}\|_\infty+N_{i_2i_1,k_1l_1}\|\bar{x}_{k_1l_1}\|_\infty\right)\cdots \left(N_{i_4i_1,k_4l_4}\|\bar{x}_{k_4l_4}\|_\infty+N_{i_1i_4,k_4l_4}\|\bar{x}_{k_4l_4}\|_\infty\right)\right]\\
 &\leq \mathbb{E}\left[\left(N_{i_1i_2,k_1l_1}+N_{i_2i_1,k_1l_1}\right)\cdots \left(N_{i_4i_1,k_4l_4}+N_{i_1i_4,k_4l_4}\right)\right] \max_{1\leq j \leq i \leq N}\|\bar{x}_{ij}\|_\infty^4\\
 &\leq 16\mathbb{E}\left[\left(N_{i_1i_2,k_1l_1}+N_{i_2i_1,k_1l_1}\right)\cdots \left(N_{i_4i_1,k_4l_4}+N_{i_1i_4,k_4l_4}\right)\right]\cdot \|x\|_\infty^4
\end{align*} 
and it yields
\begin{align*}\left|\mathbb{E}\left[\tau\left( \tr_N(A(\uY)^4)\right)\right]\right|&\leq \frac{16\|x\|_\infty^4}{4n^3}
\sum_{\substack{i_1,i_2,i_3,i_4\\k_1,k_2,k_3,k_4 \\ l_1,l_2,l_3,l_4=1}}^N
\mathbb{E}\left[\left(N_{i_1i_2,k_1l_1}+N_{i_2i_1,k_1l_1}\right)\cdots \left(N_{i_4i_1,k_4l_4}+N_{i_1i_4,k_4l_4}\right)\right].
\end{align*}
We define $X=(X_{ij})_{i,j}$ as the random matrix of Gaussian random variables
$$X_{ij}=\frac{1}{\sqrt{2}N}\sum_{k,l} \big( N_{ij,kl}+N_{ji,kl}\big) ,$$
in such a way that
\begin{align*}\left|\mathbb{E}\left[\tau\left( \tr_N(A(\uY)^4)\right)\right]\right|\leq & \frac{16\|x\|_\infty^4}{n}\!\!\!\sum_{i_1,i_2,i_3,i_4=1}^N\!\!\!\mathbb{E}\left[X_{i_1i_2}X_{i_2i_3}X_{i_3i_4}X_{i_4i_5}\right]=16\|x\|_\infty^4 \mathbb{E}\Big[\tr_N\Big(\Big(\frac{X}{\sqrt{N}}\Big)^4\Big)\Big].
\end{align*}
The term $\mathbb{E}\Big[\tr_N\Big(\Big(\frac{X}{\sqrt{N}}\Big)^4\Big)\Big]$ is the fourth moment of a GUE matrix, which is decreasing from $3$ to $2$ as $N$ is growing (see e.g. \cite[Section 1.9]{Mingo-Speicher}). We
have
$$\left|\mathbb{E}\Big[\tr_N\Big(\Big(\frac{X}{\sqrt{N}}\Big)^4\Big)\Big]\right|\leq 3\cdot 16\|x\|_\infty^4.$$
Finally,
\begin{equation}\left\|\frac{1}{g(z)}\right\|_\infty\leq  \frac{(|z|+(Nd)^{1/4}\mathbb{E}[\tau\left(\tr_N(A(\uY)^4)\right)]^{1/4})^2}{\Im(z)}\leq \sqrt{3}\cdot 4 \cdot\sqrt{Nd}\frac{ (|z|\vee \|x\|_\infty)^2}{\Im(z)}.\label{eq:step.three}\end{equation}
\paragraph{\bf{Step 4}} In order to conclude, we combine the three steps: \eqref{eq:step.one}, \eqref{eq:step.two} and \eqref{eq:step.three} yield that, for any fixed $z\in \mathbb{C}^+$, we have
$$\|g(z)-s(z)\|_{L^2(\M , \tau)}\leq 64\sqrt{6}\pi d^2\frac{(|z|\vee \|x\|_\infty\vee 1)^3}{(\Im(z)\wedge 1)^6}\frac{1}{\sqrt{N}}.$$
\end{proof}

\begin{lemma}[\cite{Tomiyama}]\label{Normtranspose}Set $\M=M_{d}(\mathbb{C})$. Then the partial transpose $\Theta:M_N(\M)\to M_N(\M)$ is bounded. More precisely, we have
$$\|\Theta(M)\|_2= \|M\|_2\ \ \text{and}\ \ \|\Theta(M)\|_\infty\leq d\cdot \|M\|_\infty.$$
\end{lemma}
Theorem~\ref{thm:Op-v-matrices-exch} will be applied in various situation, and we prove now  the following lemma which will allow passing to a smaller subalgebra, say  $\B$, and exhibiting a $\B$-valued semicircular element in the limit only under some factorization conditions on the conditional  expectation over $\B$.
\begin{lemma}\label{covcirc}
Let  $(\M, \tau)$ be a tracial $W^*$-probability space and let  $\uX=(x_{i})_{1\leq  i\leq n} $ be a family of exchangeable variables in $\M$. Set $\bar{x}_{i}=x_{i}-\mu$ where $\mu=\frac{1}{n}\sum_{1\leq i \leq n} x_{i}$ and consider their variance function 
\[
\eta_n : \M \rightarrow \M, 
\qquad 
\eta_n(b)=\frac{1}{2n}\sum_{i=1}^n  \big( \bar{x}_{i}b\bar{x}_{i}^*+\bar{x}_{i}^*b\bar{x}_{i}\big).
\] 
 Let $\mathcal{B}_n$ be the smallest $W^*$-algebra which is closed under $\eta_n$ and consider a centered semicircular variable $S_n$ over $\mathcal{B}_n$ of variance $E_{\mathcal{B}_n}(S_nbS_n)= \eta_n(b)$. We consider a von Neumann subalgebra $\mathcal{B}\subset \mathcal{B}_n$ and $E:\mathcal{M}\to \mathcal{B}$ its conditional expectation. Assume that, whenever $i,j$ are distinct, for all $b,c\in \mathcal{B}$,
 $$E(x_{i}bx_{j}^*)=E(x_{1}) bE(x_{1}^*),\quad E(x_{i}^*bx_{j})=E(x_{1}^*) bE(x_{1}),$$
 and
$$ E(x_{i}bx_{i}^*x_{k}cx_{k}^*)=E(x_{1}bx_{1}^*)E(x_{1}c x_{1}^*),\quad E(x_{i}^*bx_{i}x_{k}^*cx_{k})=E(x_{1}^*bx_{1})E(x_{1}^*c x_{1}).$$
Consider a centered semicircular variable $S$ over $\mathcal{B}$ of variance $$E(SbS)= \frac{1}{2}\Big(E(x_1^*bx_1) - E(x_1^*)bE(x_1)+E(x_1bx_1^*)-E(x_1)bE(x_1^*)\Big).$$
Then, there exists a universal constant $C>0$ such that, for all $z\in \mathbb{C}^+ $,
we have
\[
\tau [\R_{S_n}(z)] -\tau[\R_{S}(z)]\leq C \frac{ \|x_1\|_\infty^2}{\Im(z)^3}   \frac{1}{\sqrt{n}}.\]
\end{lemma}

\begin{proof}We follow the argument in the proof of \cite[Theorem 3.5]{BannaMai2021}. For all $b\in \mathcal{B}_n$, we set
$$\eta(b):= \frac{1}{2}\Big(E(x_1^*bx_1) - E(x_1^*)bE(x_1)+E(x_1bx_1^*)-E(x_1)bE(x_1^*)\Big).$$
Now, any centered semicircular variable over $\mathcal{B}_n$ of variance $\eta$ has the distribution of $S$, see Corollary 17 in \cite[Chapter 9]{Mingo-Speicher}. We first decompose $S$ and $S_n$ into a sum of $m$ variables, and at the end we take $m \rightarrow \infty$. Fix $m\in\mathbb{N}$ and let $x=\{x_j \mid 1\leq j \leq m\}$ be $\B_n$-freely independent centered semicircular variables over $\B_n$ of variance $\eta_n$. Similarly, let $y = \{y_j \mid 1 \leq j \leq m\}$ be $\B_n$-freely independent  centered semicircular variables over $\B_n$ of variance $\eta$, such that $x$ and $y$ are also free over $\B_n$. We set for any $i= 1, \ldots, n$,
 \[
{\mathbf z_i} =\frac{1}{\sqrt{m}}\sum_{j=1}^i x_j + \frac{1}{\sqrt{m}}\sum_{j=i+1}^m y_j
\qquad \text{and} \qquad 
 {\mathbf z_i^0} =\frac{1}{\sqrt{m}}\sum_{j=1}^{i-1} x_j + \frac{1}{\sqrt{m}}\sum_{j=i+1}^m y_j,
 \]
 and remark that $\mathbf z_0$ is a semicircular variable over $\B_n$ of variance $\eta$ (and consequently has the same distribution as $S$), and  $\mathbf z_m$ is a semicircular variable over $\B_n$ of variance $\eta_n$ (and consequently has the same distribution as $S_n$).
 
 We apply again the Lindeberg method in the $\B$-valued setting and notice that the families $x$ and $y$ are centered but do not have matching second moments. Therefore, from Theorem~\ref{thm:Lindeberg} we get for all $z\in\mathbb{C}^+$,
$$
   \big| \tau [\R_{S_n}(z)] -\tau[\R_{S}(z)]  \big|
=  \big|\tau \big[\R_{{\mathbf z_m}}(z)] -\tau \big[\R_{{\mathbf z_0}}(z)] \big] \big| \leq \sum_{i=1}^m \big(|P_i|+|Q_i|+|R_i|\big),
$$
where for any $i\in\{1,\ldots,n\}$, 
\begin{align*}
    P_i&=\frac{1}{\sqrt{m}}\tau\big[ \R_{\mathbf z_i^0}(z)x_i\R_{\mathbf z_i^0}(z) \big]  - \frac{1}{\sqrt{m}}\tau\big[\R_{\mathbf z_i^0}(z)y_i\R_{\mathbf z_i^0}(z)\big],\\
    Q_i&=\frac{1}{m}\tau\big[ \R_{\mathbf z_i^0}(z)x_i\R_{\mathbf z_i^0}(z)x_i\R_{\mathbf z_i^0}(z) \big]  - \tau\big[\R_{\mathbf z_i^0}(z)y_i\R_{\mathbf z_i^0}(z)y_i\R_{\mathbf z_i^0}(z)\big]\\ 
    R_i&=\frac{1}{m^{3/2}}|\tau\big[ \R_{\mathbf z_i}(z)x_i\R_{\mathbf z_i^0}(z)x_i\R_{\mathbf z_i^0}(z)x_i\R_{\mathbf z_i^0}(z) \big]  - \tau\big[\R_{\mathbf z_i}(z)y_i\R_{\mathbf z_i^0}(z)y_i\R_{\mathbf z_i^0}(z)y_i\R_{\mathbf z_i^0}(z)\big]. 
    \end{align*}
To simplify the notation, we denote by $E_{\mathcal{B}_n} [.]=\tau [.|\B]$. Note that by freeness with amalgamation over $\B_n$ and the fact that $E_{\mathcal{B}_n}[x_i]=E_{\mathcal{B}_n}[y_i]=0$, the moment-cumulant formula yields that $P_i=0$ and 
 \begin{align*}
|Q_i|&=\frac{1}{m}\Big|\tau \Big[ \R_{\mathbf z_i^0}(z) E_{\mathcal{B}_n}\big[x_i E_{\mathcal{B}_n}[\R_{\mathbf z_i^0}(z)] x_i \big] \R_{\mathbf z_i^0}(z)   - \R_{\mathbf z_i^0}(z) E_{\mathcal{B}_n}\big[y_i E_{\mathcal{B}_n}[\R_{\mathbf z_i^0}(z)]y_i \big]\R_{\mathbf z_i^0}(z)\Big] \Big|
    \\&=\frac{1}{m}\Big|\tau\big[ \R_{\mathbf z_i^0}(z)(\eta_n-\eta)(E_{\mathcal{B}_n}[\R_{\mathbf z_i^0}(z)])\R_{\mathbf z_i^0}(z) \big]\Big|
    \\&\leq \frac{1}{\Im(z)^2} \frac{1}{m}   \big\|(\eta_n-\eta)\big( E_{\mathcal{B}_n}[ R_{{\mathbf z_i^0}}(z)] \big) \big\|_{L_2}.
    \end{align*}
Finally, we have 
\[
|R_i| \leq  \frac{2\|x_1\|_\infty}{m^{3/2}} \frac{1}{\Im(z)^4}.
\]
As a consequence,
for all $z\in\mathbb{C}^+$,
$$
   \big|  \tau [\R_{S_n}(z)] -\tau[\R_{S}(z)]  \big|
\leq \frac{1}{\Im(z)^2} \frac{1}{m} \sum_{i=1}^m  \big\|(\eta_n-\eta)\big( E_{\mathcal{B}_n}[ R_{{\mathbf z_i^0}}(z)] \big) \big\|_{L_2}+ \frac{2\|x_1\|_\infty}{m^{1/2}} \frac{1}{\Im(z)^4}.
$$
Now for $b:=E_{\mathcal{B}_n}[ R_{{\mathbf z_i^0}}(z)]\in \mathcal{B}_n$, we have
\begin{align*}
2\|\eta_n(b)-\eta(b)]\|_{L^2} 
&= 2\Big\| \frac{1}{2n}\sum_{i=1}^n  (\bar{x}_{i}b\bar{x}_{i}^*+\bar{x}_{i}^*b\bar{x}_{i}) - \eta(b) \Big\| _{L^2}
\\& \leq 
\Big\| \frac{1}{n} \sum_{i=1}^n x_i b x_i^* -   E[x_1 b x_1^*] \Big\|_{L^2} + \Big\| \frac{1}{n} \sum_{i=1}^n x_i^* b x_i -   E[x_1^* b x_1] \Big\|_{L^2}\\
&\hspace{1cm}+ \big\| \mu b \mu^* -  E[x_{1}]b E[x_{1}^*] \big\| _{L^2}
+ \big\| \mu^* b \mu -  E[x_{1}^*]b E[x_{1}] \big\| _{L^2}  .
\end{align*}
We start by controlling the first terms in the last inequality and obtain
\begin{align*}
\Big\| \frac{1}{n} \sum_{i=1}^n x_i b x_i^* -   E[x_1 b x_1^*] \Big\|_{L^2}^2 
&= \tau \Big( \frac{1}{n^2} \sum_{i,j=1}^n x_i b^* x_i^* x_j b x_j^* - E[x_1b^* x_1^*] E[x_1 b x_1^*] \Big)
\\& = \tau \Big( \big(\frac{n(n-1)}{n^2} -1 \big)  E[x_1 b^* x_1^*] E[x_1 b x_1^*]  + \frac{1}{n^2} \sum_{i=1}^n x_i b^* x_i^*x_i b x_i^* \Big)
\\& \leq  \frac{2}{n} \| x_1 \|_\infty^4 \| b \|_\infty^2  \leq \frac{2\| x_1 \|_\infty^4}{n\  \Im(z)^2} ,
\end{align*}
where in the first equality we have used the property that $E(x_{i}bx_{i}^*x_{j}cx_{j}^*)=E(x_{1}bx_{1}^*)E(x_{1}cx_{1}^*)$ for $i \neq j$ and the fact that $\tau(xE[y])= \tau \circ E [x E[y]]= \tau (E[x]y)$ for any $x,y \in \M$. Therefore we get
\[
\Big\| \frac{1}{n} \sum_{i=1}^n x_i b x_i^* -   E[x_1 b x_1^*] \Big\|_{L^2}
=  \frac{\sqrt{2}\| x_1 \|_\infty^2}{\sqrt{n}\  \Im(z)} ,
\]
and similarly for the second term.
To control the third term we write 
\begin{align*}
 \big\| \mu b \mu^* -  E[x_{1}]b E[x_{1}^*] \big\| _{L^2} 
&=  \big\| (\mu - E[x_1]) b \mu^* +  E[x_{1}]b  (\mu - E[x_{1}])^*  \big\| _{L^2} 
\\ & \leq 2 \|x_1\|_\infty \|b\|_\infty \| \mu - E[x_1]\|_{L^2}\leq 
\frac{2 \|x_1\|_\infty }{\sqrt{n}\Im(z)}, 
\end{align*}
where the estimate on $\| \mu - E[x_1]\|_{L^2}$ is obtained in a similar way by using the property that $E(x_{i}bx_{j}^*)=E(x_{1}) b E(x_{1}^*)$ whenever $i\neq j$. A similar computation allows to control the last term, and we obtain that there exists a universal constant $C>0$ such that, for all $z\in \mathbb{C}^+ $,
$$\|(\eta_n-\eta)\big( E_{\mathcal{B}_n}[ R_{{\mathbf z_i^0}}(z)] \big)\|_{L_2}\leq \frac{C \|x_1\|_\infty^2}{\sqrt{n}\Im(z)}.$$
Finally,
for all $z\in\mathbb{C}^+$,
$$
   \big|  \tau [\R_{S_n}(z)] -\tau[\R_{S}(z)]  \big|
\leq \frac{1}{\Im(z)^2} \frac{1}{m} \sum_{i=1}^m  \frac{C \|x_1\|_\infty^2}{\sqrt{n}\Im(z)}+ \frac{2\|x_1\|_\infty}{m^{1/2}} \frac{1}{\Im(z)^4},
$$
and letting $m$ tend to $\infty$, we get
$$
   \big|  \tau [\R_{S_n}(z)] -\tau[\R_{S}(z)]  \big|
\leq \frac{C \|x_1\|_\infty^2}{\Im(z)^3}   \frac{1}{\sqrt{n}}.$$
\end{proof}

\section{Operator-valued matrices}\label{section:Op-v-matrices} 
For all tracial $W^*$-probability space $(\M, \tau)$  and   a finite  family of elements $\uX=(x_{ij})_{1\leq i,j\leq N} $ in $\M$, we consider the $N \times N$ operator-valued  matrix $A(\ux)$  given by 
\[
[A(\ux)]_{ij}=\frac{1}{\sqrt{2N}}(x_{ij}+x_{ji}^*).
\]
Recall that, under a condition of exchangeability, Theorem~\ref{thm:Op-v-matrices-exch} allows to approximate $A(\ux)$ by a operator-valued semicircular random variable. In this very general situation, the tracial $W^*$-probability space $(\M, \tau)$ is not specified. In this section, we review various situations where exchangeable entries appear naturally and where approximations by operator-valued semicircular random variables are possible, either as a direct corollary of Theorem~\ref{thm:Op-v-matrices-exch}, or using the Lindeberg method of the previous section.

\subsection{Free entries}\label{section:Matrixfree}
In this section, we show that a  matrix with free entries, that can possess different variances,  is close in  distribution to a  matrix  in free circular elements with the same variance structure. The latter is an operator-valued semicircular element over the subalgebra of complex-valued diagonal matrices. Moreover, we give explicit quantitative estimates for the associated Cauchy transforms on the upper complex half-plane. 
\begin{theorem}\label{theo:op-v-free}
Let  $(\M, \tau)$ be a tracial $W^*$-probability spaces. Let  $\uX=(x_{ij})_{1\leq i,j\leq N} $  and $\uY=(y_{ij})_{1\leq i,j\leq N} $ be two families of \emph{free}  variables in $\M$. Assume that 
\begin{itemize}
\item  $y_{ij}$ is a circular element for all $1\leq i,j\leq N$, 
\item  $\tau( x_{ij})=\tau( y_{ij})=0$  and $\tau(x_{ij}^* x_{ij})=\tau(y_{ij}^* y_{ij})$ for all $1\leq i,j\leq N$,
\item $\tau( x_{ij}^2)=0$ for all $1\leq i,j\leq N$.
\end{itemize}
Then for any $z\in \mathbb{C}^+$,
\[
\big| (\tau \otimes \tr_N)[\R_{A(\uX)}(z)] - (\tau \otimes \tr_N)[\R_{A(\uY)}(z)] \big|  \leq \frac{C}{\Im (z)^4}\frac{1}{\sqrt{N}} \, ,
\]
for some constant $C>0$ that depends only on $\max_{1\leq i,j\leq N} \|x_{ij}\|_\infty$.
\end{theorem}
 $A(\uY)$ is  an operator-valued semicircular element  over the subalgebra $\mathcal{D} \subset M_N (\C)$ of diagonal matrices  whose variance is given by the  completely positive map 
$
\eta : \mathcal{D} \rightarrow \mathcal{D}
$
defined by 
\[
[\eta (b)]_{ij}= \delta_{ij}\frac{1}{2N}\sum_{k=1}^N \big(\tau (x_{ik} x_{ik}^*)+\tau (x_{ki}^* x_{ki})\big) b_{kk}, \quad \text{for} \, \,  b=[b_{kl}]_{k,l=1}^N \, .
\]
One can check Chapter 9 in \cite{Mingo-Speicher} for more details on  operator-valued free probability. As a consequence of \cite[Theorem 3.3]{Ni-Di-Sp-operator-valued}, $A(\uY)$ is itself a semicircular element   in $\M$ if and only if $\eta(1)$ is a multiple of the identity. In particular, if  all the entries have the same variance, i.e. $\tau(x_{ij}^* x_{ij})=\sigma^2$ for all $1\leq i,j\leq N$, then $A(\uY)$ has a semicircular distribution.   Moreover, as the cumulative distribution function of a semicircular element is Lipschitz continuous, then by applying Theorem~\ref{BannaMai}, the above convergence can be quantified as follows in terms of the Kolmogorov distance:
 \[
 {\rm Kol} (\mu_{A(\uX)}, \mu_{A(\uY)}) \leq c N^{-1/12}.
 \]

\begin{proof}[Proof of Theorem \ref{theo:op-v-free}] 
The proof of this theorem is a direct application of Theorem \ref{theo:Lindeberg-free} where we consider the $W^*$-probability space $  (M_N(\M), \tau \otimes \tr_N)$  with $ \tau \otimes \id :  M_N(\M) \rightarrow M_N(\mathbb{C})$ as  conditional expectation over $M_N(\C)\subset  M_N(\M)$ . It is easy to see that $(\tau \otimes \id) [x_{ij} \otimes^* a_{ij}] = { (\tau \otimes \id) [y_{ij} \otimes^* a_{ij}]}=0  $ and that for any $b \in M_N(\mathbb{C})$
\begin{align*}
(\tau \otimes \id)\big[(x_{ij} \otimes^* a_{ij}) b(x_{ij} \otimes^* a_{ij}) \big]
&= ( \tau \otimes \id) \big[(y_{ij} \otimes^* a_{ij}) b(y_{ij} \otimes^* a_{ij}) \big]. 
\end{align*}
 The result follows by finally noting that $K_\infty=\mathcal{O}(1/{\sqrt{N}})$ and $K_2=\mathcal{O}(1/{N})$.  
 \end{proof}

\subsection{Infinite exchangeable entries}It is well-known by an extended de Finetti theorem \cite{Kostler}, that an \emph{infinite}   family of  identically distributed and conditionally independent elements over a tail algebra is exchangeable. In this section, we consider an infinite family of exchangeable elements and prove the convergence to an operator-valued semicircular element over the tail algebra. We  provide moreover explicit rates of convergence for the associated Cauchy transforms.

\begin{theorem}\label{pro:infinite-exch-opv} Let  $(\M, \tau)$ be a tracial $W^*$-probability space with $\M$ being  a direct integral of finite dimensional factors of dimension $\leq d$. Let $\uX=(x_{ij})_{1\leq i,j<+\infty} $ be an infinite family of  exchangeable  variables in $\M$. Let $\mathcal{M}_{tail}$ be the tail subalgebra and $E:\mathcal{M}\to \mathcal{M}_{tail}$ be its conditional expectation. Consider a centered semicircular variable $S$ over $\mathcal{M}_{tail}$ of variance 
$$E(SbS)= \frac{1}{2} \big( E \big((x_{11}-E(x_{11}))b(x_{11}-E(x_{11}))^*\big) + E \big((x_{11}-E(x_{11}))^*b(x_{11}-E(x_{11}))\big) \big).$$
Fix $N \in \mathbb{N}$ and let $A$ be the linear map defined in \eqref{def:map_matrix}.  Then  there exists a constant $C>0$ which depends on $d$ and $\|x_{11}\|_\infty$, such that for any $z \in \mathbb{C}^+$, 
\[
\big|(\tau \otimes \tr_N)[\R_{A(\uX)}(z)] -\tau[\R_{S}(z)] \big| \leq C \frac{(|z| \vee 1)^4}{(\Im(z)\wedge 1)^6}\frac{1}{\sqrt{N}}.
\]
\end{theorem}
\begin{proof}
Theorem~\ref{thm:Op-v-matrices-exch} ensures that  there exists a constant $C_1>0$ such that for any $z \in \mathbb{C}^+$, 
\[
\big|(\tau \otimes \tr_N)[\R_{A(\uX)}(z)] -\tau[\R_{S_N}(z)]\big|\leq C_1 d^{2}\frac{(|z|\vee \|x_{11}\|_\infty\vee 1)^4}{(\Im(z)\wedge 1)^6}\frac{1}{\sqrt{N}},
\]
where $S_N$ is a centered semicircular variable of variance 
$$E_{\mathcal{B}}(S_NbS_N)= \eta(b) =\frac{1}{2n}\sum_{i,j} \big(\bar{x}_{ij}b\bar{x}_{ij}^* + \bar{x}_{ij}^*b\bar{x}_{ij}\big),$$
 where $\bar{x}_{ij}= x_{ij} - \frac{1}{n} \sum_{1\leq k,l \leq N} x_{k l}$ and $n= N^2$.
By using Lemma~\ref{covcirc}, we prove that there exists a constant $C_2>0$ such that for any $z \in \mathbb{C}^+$, 
$$\big|\tau[\R_{S_N}(z)] -\tau[\R_{S}(z)] \big|\leq \frac{C_2\|x_{11}\|_\infty^2}{\Im(z)^3}\frac{1}{N}.$$
Indeed, by the conditional independence of the $x_{ij}$'s with respect to $E$, see  \cite{Kostler}, we have  for any $ b,c \in \M_{tail}$
$$E(x_{ij}bx_{kl}^*) =E(x_{11})bE(x_{11}^*), \quad 
E(x_{ij}^*bx_{kl}) =E(x_{11}^*)bE(x_{11}),  $$
and 
$$E(x_{ij}^*bx_{ij}x_{kl}^*cx_{kl})=E(x_{11}^*bx_{11})E(x_{11}^*cx_{11}), \quad
E(x_{ij}bx_{ij}^*x_{kl}cx_{kl}^*)=E(x_{11}bx_{11}^*)E(x_{11}cx_{11}^*), $$
 whenever $x_{ij}\neq x_{kl}$. 
\end{proof}

\subsection{Independent non-scalar entries}Independent and identically distributed matrices can be seen as exchangeable elements in some tracial $W^*$-probability space  $\M$ and thus many random block matrices fit nicely in the framework of operator-valued matrices with exchangeable entries.
\begin{proposition}\label{pro:ind-entries} Let  $(\M, \tau)$ be a tracial $W^*$-probability space with $\M$ being  a direct integral of finite dimensional factors of dimension $\leq d$. Let $\uX=(x_{ij})_{1\leq i,j\leq N} $ be a family of random  variables in $\M$ which are independent, identically distributed, and bounded by $K>0$. Consider a centered semicircular variable $S$ over $\mathcal{M}$ of variance 
$$b\mapsto \frac{1}{2}\big[ \mathbb{E} \big((x_{11}-\mathbb{E}(x_{11}))b(x_{11}-\mathbb{E}(x_{11}))^*\big)+\mathbb{E} \big((x_{11}-\mathbb{E}(x_{11}))^*b(x_{11}-\mathbb{E}(x_{11}))\big)\big].$$
Let $A$ be the linear map defined in \eqref{def:map_matrix}.  Then  there exists a universal constant $C>0$, which depends on $d$ and $K$, such that for any $z \in \mathbb{C}^+$,  
\[
\big|(\tau \otimes \tr_N)[\R_{A(\uX)}(z)] -\tau[\R_{S}(z)]\big| \leq C \frac{(|z|\vee 1)^4}{(\Im(z)\wedge 1)^6}\frac{1}{\sqrt{N}}.
\]
\end{proposition}

\begin{proof}Let us remark that the space $(L^{\infty}(\Omega;\mathcal{M}),\mathbb{E}\circ \tau)$ of $\mathcal{M}$-valued bounded random variables is a tracial $W^*$-probability space which is a direct integral of finite dimensional factors of dimension $\leq d$. Theorem~\ref{thm:Op-v-matrices-exch} ensures that  there exists a constant $C_1>0$ such that for any $z \in \mathbb{C}^+$, 
\[
\big|(\tau \otimes \tr_N)[\R_{A(\uX)}(z)] -\tau[\R_{S_N}(z)] \big| \leq C_1 d^{2}\frac{(|z|\vee \|x_{11}\|_\infty\vee 1)^4}{(\Im(z)\wedge 1)^6}\frac{1}{\sqrt{N}}, 
\]
where $S_N$ is a centered semicircular variable of variance $E_{\mathcal{B}}(S_NbS_N)= \eta_n(b) =\frac{1}{n}\sum_{i,j=1}^N  \bar{x}_{ij}b\bar{x}_{ij}$ where $\bar{x}_{ij}= x_{ij} - \frac{1}{n} \sum_{k,l=1}^N  x_{k l}$ and $n= N^2$.
By using Lemma~\ref{covcirc}, we prove that there exists a constant $C_2>0$ such that for any $z \in \mathbb{C}^+$, 
$$\big|\tau[\R_{S_N}(z)] -\tau[\R_{S}(z)]\big|
\leq \frac{C_2K^2}{\Im(z)^3}\frac{1}{N}. 
$$
Indeed, by the independence of the $x_{ij}$'s, we have  for any $ b,c \in \M$
 $$E(x_{ij}bx_{kl}^*)=E(x_{11}) bE(x_{11}^*),\quad E(x_{ij}^*bx_{kl})=E(x_{11}^*) bE(x_{11}),$$
 and
$$ E(x_{ij}bx_{ij}^*x_{kl}cx_{kl}^*)=E(x_{11}bx_{11}^*)E(x_{1}c x_{11}^*),\quad E(x_{ij}^*bx_{ij}x_{kl}^*cx_{kl})=E(x_{11}^*bx_{11})E(x_{1}^*c x_{11}),$$
 whenever $x_{ij}\neq x_{kl}$. 
\end{proof}

\subsection{Random matrices with i.i.d. blocks}Since i.i.d. matrices are exchangeable, random block matrices with i.i.d blocks fit instantely in our framework. We generalize the result in \cite{Girko-book} by considering matrices that are not necessarily bounded in norm and give  quantitative estimates for the Cauchy transforms. 

\begin{theorem}\label{the:i.i.dblocksKroneckerform}
Let $\{ A_{ij}=(a^{(ij)}_{kl})_{k,l=1}^d : 1\leq i, j \leq N\}$ be a family of $N^2$ independent and identically distributed $d\times d$ random matrices that are such that $\E [|a^{(ij)}_{kl}|^3]<\infty$. Let $X_N$ be   the block  matrix in $M_d(\mathbb{C})\otimes M_N(\mathbb{C})$  given by $$X_N=\frac{1}{\sqrt{2N}}\sum_{i,j=1}^N( A_{ij}\otimes E_{ij}+A_{ij}^*\otimes E_{ji}).$$  Then, as $N\to \infty$, the matrix $X_N$ has a limiting eigenvalue distribution whose Cauchy transform $g$ is determined by $g(z)=\tr_d(\R(z))$, where $\R$ is an $M_d(\mathbb{C})$-valued analytic function on $\mathbb{C}^+$  uniquely determined by the requirement that, $z\R(z)\to 1$ when $|z|\to \infty$, and that, for $z\in \mathbb{C}^+$, \begin{equation}\label{Cauchy-Master-equation}
z\R(z)=1+\eta(\R(z))\cdot \R(z),
\end{equation}
where $\eta:M_d(\mathbb{C})\to M_d(\mathbb{C})$ is the covariance mapping
$$\eta(B)=\frac{1}{2}\left(\mathbb{E}\left[(A_{11}-\mathbb{E}[A_{11}])B(A_{11}-\mathbb{E}[A_{11}])^*\right]+\mathbb{E}\left[(A_{11}-\mathbb{E}[A_{11}])^*B(A_{11}-\mathbb{E}[A_{11}])\right]\right).$$
Moreover, there exists a constant $C>0$ which depends on the distribution of $A_{11}$, and such that for any $z \in \mathbb{C}^+$,
\[
\big|\mathbb{E}[\tr_{dN}(\R_{X_N}(z))] - g(z)\big|
\leq C\frac{(|z| \vee 1)^4}{(\Im(z) \wedge 1)^6}\frac{1}{\sqrt{N}}.
\]
\end{theorem}
\begin{proof}
We remark again that the blocks $A_{ij}$ are i.i.d. and thus exchangeable. However, our Proposition~\ref{pro:ind-entries}  is only applicable if the block matrices are  bounded which is not  necessarily the case here. In order to apply our results,  we will approximate first $X_N$ by a  matrix  $Y_N$ where the blocks $B_{ij}=(b^{(ij)}_{kl})_{k,l=1}^d : 1\leq j \leq i \leq N$ are i.i.d. matrices that are bounded in norm by some $K>0$ and have the same mean and  covariance structure as the $A_{ij}$'s.
i.e. $\text{Cov}(b^{(ij)}_{pq},b^{(ij)}_{rs})= \text{Cov}(a^{(ij)}_{pq},a^{(ij)}_{rs})$ (this can be done for example by taking a linear map decorrelating the $a^{(ij)}_{kl}$'s and replacing the corresponding random variables by independent two-point distributions). Note that  $X_N$ and $Y_N$ have respectively the same distribution as
\[
 \sum_{1\leq i,j\leq N} A_{ij}\otimes^* a_{ij}
\quad \text{and} \quad
\sum_{1\leq i,j\leq N}  B_{ij}\otimes^* a_{ij}, 
\]
where $ a_{ij}= \frac{1}{ \sqrt{2N}} E_{ij}$ and $(E_{ij})_{1\leq i,j \leq N}$ are the canonical $N\times N$ matrices. Let us now denote respectively  by  $(x_1,\dots,x_n)$  and  $(y_1,\dots,y_n)$ the $n$-tuples $\{A_{ij}\otimes^* a_{ij} \, :  \, 1\leq i,j \leq N \}$ and  $\{B_{ij}\otimes^* a_{ij} \, :  \, 1\leq  i,j \leq N \}$ of  self-adjoint random  elements in $\M$. Setting for any $i \in \{1, \ldots ,n \}$, 
$$
\mathbf z_i = x_1 +\cdots + x_i+y_{i+1}+ \cdots+ y_{n}
$$
and $$ \mathbf z_i^0 = x_1 +\cdots + x_{i-1}+y_{i+1}+ \cdots + y_{n},
$$
then  $X_N$ and $Y_N$ have the same distributions as $\mathbf z_{n}$ and  $\mathbf z_{0}$ respectively. Therefore, applying the noncommutative Lindeberg method in  Theorem~3.1, we get  for any   $z\in \mathbb{C}^+$,
\[
\big|\mathbb{E}[\tr_{dN}(\R_{X_N}(z))] - \mathbb{E}[\tr_{dN}(\R_{Y_N}(z))] \big| \leq \sum_{i=1}^{n} (|\mathbb{E}[P_i]|+|\mathbb{E}[Q_i]|+|\mathbb{E}[T_i]|).
\]
where, for any $i \in \{1, \ldots ,{n} \}$, 
\[
P_i= \tr_{d}\otimes \tr_N\big[\R_{\mathbf z_i^0}(z)(x_i-y_i)\R_{\mathbf z_i^0}(z) \big],
\]
$$
Q_i=  \tr_{d}\otimes \tr_N\big[ \R_{\mathbf z_i^0}(z)x_i\R_{\mathbf z_i^0}(z)x_i\R_{\mathbf z_i^0}(z) \big]  - \tr_{d}\otimes \tr_N\big[\R_{\mathbf z_i^0}(z)y_i\R_{\mathbf z_i^0}(z)y_i\R_{\mathbf z_i^0}(z)\big],
$$
and
$$
T_i=  \tr_{d}\otimes \tr_N\big[\R_{\mathbf z_i}(z) \big(x_i \R_{\mathbf z_i^0}(z) \big)^3
\big] 
- \tau\big[\R_{\mathbf z_{i-1}}(z) \big(y_i\R_{\mathbf z_i^0}(z) \big)^3
\big] .
$$
As the $B_{ij}$'s are independent and as the first and second moments of their entries match with those of the $A_{ij}$'s, we get that $\mathbb{E}[P_i]=0$ and $\mathbb{E}[Q_i]=0$. Furthermore, we have for any $i \in \{1, \ldots , n \}$
\begin{align*}|\mathbb{E}[T_i]|&\leq \big|\mathbb{E}\tr_{d}\otimes \tr_N\big[\R_{\mathbf z_i}(z) \big(x_i \R_{\mathbf z_i^0}(z) \big)^3\big]\big|+\big|\mathbb{E}\tr_{d}\otimes \tr_N\big[\R_{\mathbf z_i}(z) \big(y_i \R_{\mathbf z_i^0}(z) \big)^3\big]\big|\\
&\leq  \frac{1}{\Im(z)^4}  \mathbb{E}\tr_{d}\otimes \tr_N\big[|x_i|^3\big]+  \frac{1}{\Im(z)^4}\mathbb{E}\tr_{d}\otimes \tr_N\big[|y_i|^3\big]
\leq \frac{C_1}{\Im(z)^4}\frac{1}{N^{5/2},}
\end{align*}
for some positive constant $C_1$. Finally, summing over $i =1, \ldots , N^2$, we get
\[
\big|\mathbb{E}[\tr_{dN}(\R_{X_N}(z))] - \mathbb{E}[\tr_{dN}(\R_{Y_N}(z))] \big| \leq \frac{C_1}{\Im(z)^4}\frac{1}{\sqrt{N}}.
\]
The study is thus reduced to $Y_N$ and Proposition~\ref{pro:ind-entries} now  directly ensures that there exists a constant $C_2>0$ which depends on the distributino of $A_{11}$, such that  
\[
\big|\mathbb{E}\tr_d \otimes \tr_N[\R_{Y_N}(z)] -\tau[\R_{S}(z)] \big|\leq C_2\frac{(|z|\vee  1)^4}{(\Im(z)\wedge 1)^6}\frac{1}{\sqrt{N}},
\]
where $S$ is a centered semicircular variable over  the subalgebra of $d \times d$ random matrices  of covariance \[\eta(B)=\frac{1}{2}\left(\mathbb{E}\left[(A_{11}-\mathbb{E}[A_{11}])B(A_{11}-\mathbb{E}[A_{11}])^*\right]+\mathbb{E}\left[(A_{11}-\mathbb{E}[A_{11}])^*B(A_{11}-\mathbb{E}[A_{11}])\right]\right).\qedhere\]
\end{proof}

\subsection{Random matrices with correlated  blocks}
It is also possible to consider random block  matrices in which the  blocks are correlated. We consider in the following theorem a general matrix in which the dependency comes from the correlations between different blocks and the possibility that some blocks might appear more than once.    

\begin{theorem}\label{theo:Kronecker-correlations}
Let $d\geq 1$ and $L\geq 1$ and consider the $Nd \times Nd$  block matrix $X_N$ defined by
\[
X_N = \sum_{k=1}^L (\beta_k \otimes Y_k + \beta_k^* \otimes Y_k^*),
\]
where  the $\beta_k $ are $d \times d$ deterministic matrices  and the $Y_k=\frac{1}{\sqrt{N}}(y_{ij}^{(k)})_{i,j=1}^N$  are $N \times N$ random matrices such that 
\begin{itemize}
\item the $N^2$ families $\{y_{ij}^{(1)}, \ldots , y_{ij}^{(L)}\}$ (with $1\leq i,j \leq N$) are i.i.d.
\item $\text{Cov} \big(y_{ij}^{(k)}, \overline{y_{pr}^{(l)}}\big)= \delta_{ip}\delta_{jr} \sigma({k,l}) $
\item $\E [|y_{ij}^{(k)}|^3] <\infty$,
\end{itemize}
where $\sigma$ is such that $\sigma (k,l)= \overline{\sigma (l,k)}$. Then, for $N\to \infty$, the matrix $X_N$ has a limiting eigenvalue distribution whose Cauchy transform $g$ is determined by $g(z)=\tr_d(\R(z))$, where $\R$ is an $M_d(\mathbb{C})$-valued analytic function on $\mathbb{C}^+$, which is uniquely determined by the requirement that, $z\R(z)\to 1$ when $|z|\to \infty$, and that, for $z\in \mathbb{C}^+$, 
$$z\R(z)=1+\eta(\R(z))\cdot \R(z),$$
where $\eta:M_d(\mathbb{C})\to M_d(\mathbb{C})$ is the covariance mapping
\[
\eta(B)=\frac{1}{L^2} \sum_{k,l=1}^L \Big(\sigma({k,l} ) \beta_k B \beta_l^* + \overline{\sigma({k,l})}  \beta_k^* B \beta_l \Big) .
\]
Moreover, there exists a constant $C>0$ which depends on the distribution of $\{y_{11}^{(1)}, \ldots , y_{11}^{(L)}\}$ and on $\{\beta_1, \ldots , \beta_L\}$, such that for any $z \in \mathbb{C}^+ $,
\[
\big|\mathbb{E}[\tr_{dN}(\R_{X_N}(z))] - g(z)\big|
\leq C\frac{(|z|\vee 1)^4}{(\Im(z) \wedge 1)^6}\frac{1}{\sqrt{N}}.
\]
\end{theorem}

The global law and location of the spectrum of this general model have been investigated  under the name of Kronecker  matrices in \cite{Alt-al-Kronecker-19}, when no correlations are allowed between the $Y_k$'s and under boundedness conditions on the moments of the $y_{ij}^{(k)}$'s and on the spectral norm of the $\beta_k$'s.  Here, we allow such correlations between the $Y_k$'s and we relax the conditions on the moments of the $y_{ij}^{(k)}$'s. However, the results in \cite{Alt-al-Kronecker-19} also cover the non-Hermitian case. 

Matrices with decaying correlations  have been analysed  in~\cite{Erdos-al-19} where, under an additional restrictive moment condition,  the speed of convergence was shown to be $N^{\varepsilon-1}$ for all $\varepsilon >0$  with overwhelming probability. However, we are looking for a universal approach that covers matrices with exchangeable entries in  more general von Neumann algebras. With this generality and under  mild moment conditions, we can't  achieve  a better speed of convergence.

\begin{proof}[Proof of Theorem \ref{theo:Kronecker-correlations}]
It is enough to notice that $X_N= \frac{1}{\sqrt{2}}(x_{ij} + x_{ij}^*)_{i,j=1}^N$ with $x_{ij}= \sqrt{2} \sum_{k=1}^L y_{ij}^{(k)} \beta_k$ and then apply Theorem  \ref{the:i.i.dblocksKroneckerform}.
\end{proof}

\section{Operator-valued Wigner matrices} \label{section:Wigner}
This section is devoted for the study  of Wigner matrices with operator-valued  entries. Let  $(\M, \tau)$ be a tracial $W^*$-probability space and  $\uX=(x_{ij})_{1\leq j\leq i\leq N} $ be a finite  family of elements in $\M$. An operator-valued Wigner matrix $W(\ux)$ of dimension $N \times N$ is the matrix given by 
\[
[W(\ux)]_{ij}=\frac{1}{\sqrt{N}}\left\{
\begin{array}
[c]{ll}
x_{ij} & \mbox{ if }j< i,\\
x_{ji}^* & \mbox{ if }j>i,\\
(x_{ii}+x_{ii}^*)/2 & \mbox{ if }i=j.\\
\end{array} \, 
\right.  
\]
To study its behavior, we  consider the Cauchy transform and apply the results of the previous section.  With this aim, we set $n=N(N+1)/2$ and note that the associated linear map $W : \M^n \rightarrow M_N(\M)$ can be written as follows:
\begin{equation}\label{def:map_Wigner}
W(\ux) = \sum_{1\leq j \leq i \leq N} \big(x_{ij} \otimes a_{ij}+ x_{ij}^* \otimes a_{ij}^* \big)=:\sum_{1\leq j \leq i \leq N} x_{ij} \otimes^* a_{ij},
\end{equation}
where $a_{ii}=  \frac{1}{2 \sqrt{N}} E_{ii}$ and   $ a_{ij}= \frac{1}{\sqrt{N}} E_{ij}$ for  $j<i$ and $(E_{ij})_{1\leq i,j \leq N}$ are the canonical $N\times N$ matrices. 
\subsection{Free entries}\label{section:wignerfree}
 Voiculescu proved in \cite{Voiculescu90} that, for any $N\geq 1$, an $N\times N$ Wigner matrix with free circular off-diagonal entries and free semicircular diagonal entries, is itself semicircular.  This result was later generalized in \cite[Corollary 3.3]{Ni-Sh-Sp-02} to matrices with free self-adjoint elements on the diagonal and  free $R$-diagonal elements (introduced in \cite{Ni-Sp-Rdiagonal}) otherwise. The authors give, for any $N\geq 1$,  the distribution of the matrix in terms of its $R$-transform under some summation conditions on the free cumulants of the  entries of each row.  In general, one needs to take the limit as $N\rightarrow\infty$ to give an explicit  description of  the \emph{limiting} distribution  of the matrix. For the case of free identically distributed entries, a semicircular element appears in the limit  \cite{Ryan-98, Liu2018}. 

Again with similar proofs as in Section \ref{section:Op-v-matrices},  one can show that a Wigner matrix having free entries with possible different variances, is close in  distribution to a  matrix  in free circular/semicircular elements with the same variance structure. The latter is an operator-valued semicircular element over the subalgebra of complex-valued diagonal matrices. Moreover, we give explicit quantitative estimates for the difference of the associated Cauchy transforms on the upper complex half-plane. 
\begin{theorem}\label{theo:wignerfree}
Let  $(\M, \tau)$ be a tracial $W^*$-probability spaces. Let  $\uX=(x_{ij})_{1\leq j\leq i\leq N} $  and $\uY=(y_{ij})_{1\leq j\leq i\leq N} $ be two families of \emph{free}  variables in $\M$. Assume that 
\begin{itemize}
\item   $x_{ii}^*=x_{ii}$ for all $i$,
\item  $y_{ii}$ is a  semicircular element for all $i$ and $y_{ij}$ is a circular element for all $j< i$, 
\item  $\tau( x_{ij})=\tau( y_{ij})=0$  and $\tau(x_{ij}^* x_{ij})=\tau(y_{ij}^* y_{ij})$ for all $j\leq i$,
\item $\tau( x_{ij}^2)=0$ for all $j< i$.
\end{itemize}
Then for any $z\in \mathbb{C}^+$,
\[
\big| (\tau \otimes \tr_N)[\R_{W(\uX)}(z)] - (\tau \otimes \tr_N)[\R_{W(\uY)}(z)] \big|  \leq \frac{C}{\Im (z)^4}\frac{1}{\sqrt{N}} \, ,
\]
for some constant $C>0$ that depends only on $\max_{1\leq j \leq i \leq N} \|x_{ij}\|_\infty$.
\end{theorem}
 $W(\uY)$ is  an operator-valued semicircular element  over the subalgebra $\mathcal{D} \subset M_N (\C)$ of diagonal matrices  whose variance is given by the  completely positive map 
$
\eta : \mathcal{D} \rightarrow \mathcal{D}
$
defined by 
\[
[\eta (b)]_{ij}= \frac{1}{N}\sum_{k=1}^N \tau ( x_{ik} x_{jk}^*) b_{kk}, \quad \text{for} \, \,  b=[b_{kl}]_{k,l=1}^N \, .
\]
As a consequence of \cite[Theorem 3.3]{Ni-Di-Sp-operator-valued}, $W(\uY)$ is itself a semicircular element 
 in $\M$ if and only if  the sums of variances in each row of the matrix are the same. In particular, if  all the entries have the same variance, i.e. $\tau(x_{ij}^* x_{ij})=\sigma^2$ for all $j \leq i$, then $W(\uY)$ has a semicircular distribution. Again, as the cumulative distribution function of a semicircular element is Lipschitz continuous, then by applying Theorem~\ref{BannaMai}, the above convergence can be quantified as follows in terms of the Kolmogorov distance:
 \[
 {\rm Kol} (\mu_{A(\uX)}, \mu_{A(\uY)}) \leq c N^{-1/12}.
 \]

\begin{remark}
Another relevant question is when, for any $N\geq 1$, an $N \times N$ operator-valued Wigner matrix has itself the same distribution as its entries.  Indeed this case exists and it was shown by Shlyakhtenko \cite{Shlyakhtenko97} that a Wigner matrix in $M_N(\M)$ with free creation operators as entries is itself a creation operator in $\M$. Later, Ryan \cite[Corollary 29]{Ryan-98} showed that Shlyakhtenko's example is close to be the only one. Otherwise he showed that the same distribution can be only obtained when taking the limit as $N \rightarrow \infty$.
\end{remark}

\subsection{Random matrices with i.i.d. blocks}
We generalize  results in \cite{Girko-book} by considering matrices that are not necessary bounded in norm and give  quantitative estimates for the Cauchy transforms. 
\begin{theorem}\label{the:i.i.dblocksKroneckerform2}
Let $\{ A_{ij}=(a^{(ij)}_{kl})_{k,l=1}^d : 1\leq  j\leq i \leq N\}$ be a family of independent $d\times d$ random matrices that are such that  $\E [|a^{(ij)}_{kl}|^3]<\infty$. We assume that $\{ A_{ij}\}_{j< i}$ are identically distributed, and identically distributed to $\{ A_{ij}^*\}_{j< i}$. Let $X_N$ be   the block  matrix in $M_d(\mathbb{C})\otimes M_N(\mathbb{C})$  given by $$X_N=\frac{1}{\sqrt{N}}\sum_{1\leq  j\leq i \leq N}( A_{ij}\otimes E_{ij}+A_{ij}^*\otimes E_{ji}).$$  Then, as $N\to \infty$, the matrix $X_N$ has a limiting eigenvalue distribution whose Cauchy transform $g$ is determined by $g(z)=\tr_d(\R(z))$, where $\R$ is an $M_d(\mathbb{C})$-valued analytic function on $\mathbb{C}^+$  uniquely determined by the requirement that, $z\R(z)\to 1$ when $|z|\to \infty$, and that, for $z\in \mathbb{C}^+$, 
$$z\R(z)=1+\eta(\R(z))\cdot \R(z),$$
where $\eta:M_d(\mathbb{C})\to M_d(\mathbb{C})$ is the covariance mapping
$$\eta(B)=\mathbb{E}\left[(A_{21}-\mathbb{E}[A_{21}])B(A_{21}-\mathbb{E}[A_{21}])^*\right].$$
Moreover, there exists $C>0$ which depends on the distribution of $A_11$ such that for any $z \in \mathbb{C}^+ $, we have
\[
\big|\mathbb{E}[\tr_{dN}(\R_{X_N}(z))] - g(z)\big|
\leq C\frac{(|z|\vee 1)^4}{(\Im(z) \wedge 1)^6}\frac{1}{\sqrt{N}}.
\]

\end{theorem}
\begin{proof}
Let us consider a family of independent and identically distributed $d\times d$ random matrices $\{ B_{ij}=(b^{(ij)}_{kl})_{k,l=1}^d : 1\leq i, j \leq N\}$ which are distributed as the matrices $\{ A_{ij},A_{ij}^*\}_{j< i}$. We will show that the distribution of $X_N$ is closed to the distribution of
$$ Y_N=\frac{1}{\sqrt{2N}}\sum_{i,j=1}^N( B_{ij}\otimes E_{ij}+B_{ij}^*\otimes E_{ji}).$$
We set $x(i,j)=\frac{1}{\sqrt{N}}(A_{ij}\otimes E_{ij}+A_{ij}^*\otimes E_{ji})$ for $j\leq i$, $$y(i,j)=\frac{1}{\sqrt{2N}}\left(B_{ij}\otimes E_{ij}+B_{ij}^*\otimes E_{ji}+B_{ji}^*\otimes E_{ij}+B_{ji}\otimes E_{ji}\right),$$ for $j< i$ and $y(i,i)=\frac{1}{\sqrt{2N}}\left(B_{ii}\otimes E_{ii}+B_{ii}^*\otimes E_{ii}\right)$. 
Let us now denote respectively  by  $(x_1,\dots,x_n)$  and  $(y_1,\dots,y_n)$ the $n$-tuples $\{x(i,j) \, :  \, 1\leq j\leq i \leq N \}$ and  $\{y(i,j) \, :  \, 1\leq  j\leq i \leq N \}$ of  self-adjoint random  elements in $\M$, with $n=N(N+1)/2$. Setting for any $i \in \{1, \ldots ,n \}$, 
$$
\mathbf z_i = x_1 +\cdots + x_i+y_{i+1}+ \cdots+ y_{n}
$$
and $$ \mathbf z_i^0 = x_1 +\cdots + x_{i-1}+y_{i+1}+ \cdots + y_{n},
$$
then  $X_N$ and $Y_N$ are respectively $\mathbf z_{n}$ and  $\mathbf z_{0}$. Therefore, applying the noncommutative Lindeberg method in  Theorem \ref{thm:Lindeberg}, we get  for any   $z\in \mathbb{C}^+$,
\[
\big|\mathbb{E}[\tr_{dN}(\R_{X_N}(z))] - \mathbb{E}[\tr_{dN}(\R_{Y_N}(z))] \big| \leq \sum_{i=1}^{n} (|\mathbb{E}[P_i]|+|\mathbb{E}[Q_i]|+|\mathbb{E}[T_i]|).
\]
where, for any $i \in \{1, \ldots ,{n} \}$, 
\[
P_i= \tr_{d}\otimes \tr_N\big[\R_{\mathbf z_i^0}(z)(x_i-y_i)\R_{\mathbf z_i^0}(z) \big],
\]
$$
Q_i=  \tr_{d}\otimes \tr_N\big[ \R_{\mathbf z_i^0}(z)x_i\R_{\mathbf z_i^0}(z)x_i\R_{\mathbf z_i^0}(z) \big]  - \tr_{d}\otimes \tr_N\big[\R_{\mathbf z_i^0}(z)y_i\R_{\mathbf z_i^0}(z)y_i\R_{\mathbf z_i^0}(z)\big],
$$
and
$$
T_i=  \tr_{d}\otimes \tr_N\big[\R_{\mathbf z_i}(z) \big(x_i \R_{\mathbf z_i^0}(z) \big)^3
\big] 
- \tau\big[\R_{\mathbf z_{i-1}}(z) \big(y_i\R_{\mathbf z_i^0}(z) \big)^3
\big] .
$$
As the $B_{ij}$'s are independent and as the first and second moments of their entries match with those of the $A_{ij}$'s, we get that
$$\sum_{i=1}^n|\mathbb{E}[P_i]|\leq \frac{1}{\Im(z)^2}\sum_{i=1}^N \mathbb{E}[\|x(i,i)\|_{L_1}+\|y(i,i)\|_{L_1}] \leq \frac{C}{(\Im(z) \wedge 1)^2} \frac{1}{\sqrt{N}},$$ and $$\sum_{i=1}^n|\mathbb{E}[Q_i]|\leq \frac{1}{\Im(z)^3}\sum_{i=1}^N \mathbb{E}[\|x(i,i)\|_{L_2}^2+\|y(i,i)\|_{L_2}^2]\leq \frac{C}{(\Im(z) \wedge 1)^3}\frac{1}{N}$$
where $C$ is a positive constant that may change from one line to the other. Furthermore, we have for any $i \in \{1, \ldots , n \}$
\begin{align*}|\mathbb{E}[T_i]|&\leq \left|\mathbb{E}\tr_{d}\otimes \tr_N\big[\R_{\mathbf z_i}(z) \big(x_i \R_{\mathbf z_i^0}(z) \big)^3\big]\right|+\left|\mathbb{E}\tr_{d}\otimes \tr_N\big[\R_{\mathbf z_i}(z) \big(y_i \R_{\mathbf z_i^0}(z) \big)^3\big]\right|\\
&\leq  \frac{1}{\Im(z)^4}  \mathbb{E}\tr_{d}\otimes \tr_N\big[|x_i|^3\big]+  \frac{1}{\Im(z)^4}\mathbb{E}\tr_{d}\otimes \tr_N\big[|y_i|^3\big]
\leq \frac{C}{(\Im(z) \wedge 1)^4}\frac{1}{N^{5/2}}.
\end{align*}
Putting the above bounds together, we finally get
\[
\big|\mathbb{E}[\tr_{dN}(\R_{X_N}(z))] - \mathbb{E}[\tr_{dN}(\R_{Y_N}(z))] \big| \leq \frac{C}{(\Im(z) \wedge 1)^4}\frac{1}{\sqrt{N}}.
\]
The study is thus reduced to $Y_N$ and Theorem~\ref{the:i.i.dblocksKroneckerform} now  directly ensures that there exists  $C>0$ such that for any $z \in \mathbb{C}^+$ , we have
\[\big|\mathbb{E}[\tr_{dN}(\R_{Y_N}(z))] -g(z)\big|
\leq C\frac{(|z|\vee 1)^4}{(\Im(z) \wedge 1)^6}\frac{1}{\sqrt{N}}.
\qedhere\]
\end{proof}
\subsection{Random matrices with correlated  blocks}
 We consider in Theorem \ref{theo:Correlated-Gau} matrices with a Wigner block structure in which the blocks can be correlated but their entries are i.i.d. random variables having finite third moments.  This generalizes  the results in~\cite{Ra-Or-Br-Sp-08} in which the limiting distribution was described in the Gaussian case and without any rates of convergence. We then give in Theorem \ref{theo:circulant} a concrete example on circulant block matrices for which the distribution of the operator-valued semicircular can be explicitly computed and for which the rate of convergence of the Cauchy transform  implies in turn  convergence in Kolmogorov distance.  This generalizes Proposition~1 in  \cite{Oraby2007}. 

\begin{theorem}\label{theo:Correlated-Gau} Fix $d\geq 1$ and consider the $Nd \times Nd$  block matrix $X_N$ defined by
$$X_N= \frac{1}{\sqrt{d}}\left(\begin{array}{ccc}
A^{(11)} & \cdots& A^{(1d)} \\ 
\vdots& \ddots&\vdots \\ 
A^{(d1)} & \cdots & A^{(dd)}
\end{array} \right), $$ 
where, for any $1\leq j \leq i \leq d $, $A^{(ji)}= (A^{(ij)})^*$ and $A^{(ij)}=(\frac{1}{\sqrt{N}}a_{rp}^{(ij)})_{r,p=1}^N$ are $N\times N$  random matrices having  i.i.d entries above the diagonal and such that  for any $i,j,k,l =1, \ldots, d$ 
$$ 
\text{Cov}(a_{rp}^{(ij)},a_{qs}^{(kl)})=\delta_{pq}\delta_{rs}\sigma(i,j;k,l)\quad \text{and} \quad   \E[|a_{11}^{(ij)}|^3]< \infty,$$
where $\sigma$ is such that $\sigma(i,j;k,l)=\overline{\sigma(k,l;i,j)}$. Then, for $N\to \infty$, the matrix $X_N$ has a limiting eigenvalue distribution whose Cauchy transform $g$ is determined by $g(z)=\tr_d(\R(z))$, where $\R$ is an $M_d(\mathbb{C})$-valued analytic function on $\mathbb{C}^+$, which is uniquely determined by the requirement that, $z\R(z)\to 1$ when $|z|\to \infty$, and that, for $z\in \mathbb{C}^+$, 
$$z\R(z)=1+\eta(\R(z))\cdot \R(z),$$
where $\eta:M_d(\mathbb{C})\to M_d(\mathbb{C})$ is the covariance mapping
$$\eta(B)_{ij}=\frac{1}{d} \sum_{k,l=1}^d\sigma(i,k;l,j)B_{kl}. $$
Moreover, there exists $C>0$ which depends on the distribution of $( a_{11}^{(ij)})_{1\leq j\leq i \leq d}$ such that for any $z \in \mathbb{C}^+$,
\[
\big|\mathbb{E}[\tr_{dN}(\R_{X_N}(z))] - g(z)\big|
\leq C\frac{(|z|\vee 1)^4}{(\Im(z) \wedge 1)^6}\frac{1}{\sqrt{N}}.
\]
\end{theorem}
We recall that when the cumulative distribution function of the operator-valued semicircular element (whose Cauchy transform is $g$) is H\"older continuous, the estimate on the Cauchy transform gives quantitative estimates on the Kolmogorov distance by Theorem~\ref{BannaMai}.
\begin{proof}Note that $X_N=\frac{1}{\sqrt{d}}\sum_{i,j}A^{(ij)}\otimes E_{ij}
= \frac{1}{\sqrt{N}}\sum_{r,p}E_{rp}\otimes A_{rp}$
where $A_{rp} = A_{pr}^*$ and  the blocks $A_{rp}= \frac{1}{\sqrt{d}}( a_{rp}^{(ij)})_{1\leq j\leq i \leq d}$ are  i.i.d.  $d\times d$ random matrices hence exchangeable.  The proof is then a direct application of Theorem~\ref{the:i.i.dblocksKroneckerform2}.
\end{proof}

We give another possible application of our results to self-adjoint block circulant matrices.   Let  $(\M, \tau)$ be a tracial $W^*$-probability space and  $ x_1, \ldots, x_d$ be a family of elements in $\M$ such that $x_i=x_{d-i+2}$ for $i=2, 3, \ldots, d$. An operator-valued $d \times d$ self-adjoint circulant matrix over $ (x_1, \ldots, x_d)$ is the matrix of the form
\[
C_d(x_1, \ldots, x_d):=\frac{1}{\sqrt{d}}\left(\begin{array}{cccc}
x_1 & x_2 
&\cdots& x_d \\ 
x_d & x_1
&\cdots& x_{d-1} \\ 
\vdots&\vdots 
 & \ddots&\vdots \\ 
x_2 & x_3 
&\cdots & x_1
\end{array} \right)
. 
\] 
We are interested in the case of $Nd \times Nd$ self-adjoint block matrices with circulant structure in which the blocks are, up to symmetry, independent and identically distributed $N \times N$ Wigner matrices. For this, let $(a_{ij})_{i,j \geq1}$ be a family of independent and identically distributed random variables such that $\E(a_{ij}^2)=0$, $\E(|a_{ij}|^2)=1$ and $\E(|a_{ij}|^4)< \infty$. Recall the mapping in \eqref{def:map_Wigner} and let $A$ be the $N \times N$ Wigner matrix given by $A=W((a_{ij})_{1\leq j \leq i \leq N})$. 
Let $A^{(1)}, \ldots, A^{(\lfloor \frac{d}{2}\rfloor +1)}$ be independent copies of the Wigner matrix $A$ and then set $A^{(i)}=A^{(d-i+2)}$ for $i= \lfloor \frac{d}{2}\rfloor +1, \ldots , d$. Denote by $X_N$ the $Nd \times Nd $ self-adjoint block circulant matrix given by  $X_N := C_d( A^{(1)}, \ldots, A^{(d)})$. 

It has been shown in \cite{Oraby2007} that the limiting spectral distribution of $X_N$ is a weighted sum of two semicircular distributions, i.e. has an $SS$-law as named by Girko in \cite{Girko2000}. We  provide now a quantitative version of this result: 
\begin{theorem}\label{theo:circulant}
For fixed $d \geq 1$,  there exists $C>0$ which depends on the distribution of $a_{11}$ such that for any $z \in \mathbb{C}^+ $,
 \[
\big|\mathbb{E}[\tr_{dN}(\R_{X_N}(z))] - \tau ( \R_d(z)) \big|
\leq C\frac{(|z|\vee 1)^4}{(\Im(z) \wedge 1)^6}\frac{1}{\sqrt{N}} ,
\]
where $ \tau \circ \R_d$ is the Cauchy transform associated with the distribution $\mu_d$ given by
\[
\mu_d=  \left\{
\begin{array}
[c]{ll}
\frac{d-1}{d} \gamma_{\frac{d-1}{d}} + \frac{1}{d} \gamma_{\frac{2d-1}{d}} & \mbox{ if $d$ is odd, }\\
\frac{d-2}{d} \gamma_{ \frac{d-2}{d}} + \frac{2}{d} \gamma_{ \frac{2d-2}{d}} & \mbox{ if $d$ is even,}\\
\end{array} \, 
\right.  
\]
and $\gamma_{\sigma^2}$ denotes the centered semicircular distribution of variance $\sigma^2$. Furthermore, the averaged distribution of $X_N$ converges to $\mu_d$ in Kolmogorov distance with 
\[
{\rm Kol} (\mathbb{E} \mu_{X_N} , \mu_d) \leq c N^{-1/24}, \qquad \text{for some positive constant } c.
\]
\end{theorem}

\begin{proof}

Let us first introduce some notation. Consider the family of $d \times d$ self-adjoint circular random matrices $\{A_{ij} \ : i,j \geq 1 \}$ such that $A_{ij}=A_{ji}=A_{ij}^*=A_{ji}^*$ and $A_{ij}= C_d ( a_{ij}^{(1)} , a_{ij}^{(2)} , \ldots ,a_{ij}^{(d)} )$. 
The matrix $X_N$ is   the block  matrix in $M_d(\mathbb{C})\otimes M_N(\mathbb{C})$  given by $$X_N=\frac{1}{\sqrt{dN}}\sum_{1\leq  j< i \leq N}( A_{ij}\otimes E_{ij}+A_{ij}\otimes E_{ji})+\frac{1}{\sqrt{dN}}\sum_{1\leq  i \leq N} A_{ii}\otimes E_{ii}.$$ Note that for any $1 \leq j \leq i \leq N$, the matrices $A_{ij}$ are independent and identically distributed. Theorem~\ref{the:i.i.dblocksKroneckerform2} allows us to write that there exists $C>0$ such that for any $z\in \mathbb{C}^+$, 
$$\big|\tau[\R_{X_N}(z)] -\tau[\R_{S}(z)]\big|
\leq C\frac{(|z|\vee 1)^4}{(\Im(z) \wedge 1)^6}\frac{1}{\sqrt{N}},$$
where $S$ is a centered semicircular variable over $M_d(\mathbb{C})$ of covariance mapping
\[
\eta: M_d(\C) \rightarrow M_d(\C), \quad \eta(B)= \frac{1}{d}\E [A_{12} B A_{12}].
\]
One can directly show that $S$ has distribution $\mu_d$ or alternatively follow similar computations as in the proof of Proposition 1 in  \cite{Oraby2007}. The convergence in Kolmogorov distance now follows from by Theorem~\ref{BannaMai} since the cumulative distribution function of a sum of semicirculars ($SS$-law) is Lipschitz continuous.
\end{proof}

\section{Operator-valued Wishart matrices}\label{section:Wishart}
Let $(\M, \tau)$ be a tracial $W^*$-probability space and let $\mathbf{h}=(h_{ij})_{1\leq i,j \leq N}$ be a family of elements in $\M$. Consider the $N \times N$ matrix $H$ in $(M_N(\M) , \tau \otimes \tr_N)$ given by 
\[
H:=H_N= \left(\begin{array}{ccc}
h_{11} & \cdots & h_{1N} \\ 
\vdots&  &\vdots \\ 
h_{N1} & \cdots & h_{NN}
\end{array} \right) \in M_N (\M) \,.
\]
We are interested in studying the $N \times N$ operator-valued Wishart matrix $\frac{1}{N} HH^*$. This is an interesting  model in random matrix theory and is very well understood in the classical setting of commutative entries. However, not much is known  in the operator-valued case despite its significance importance  for many applications.  We mention, among others, wireless communication systems whose channels have such a Hermitian block structure \cite{Ra-Or-Br-Sp-08}. Although our result is of different nature, we  mention the work of  Bryc  \cite{Bryc2008} in which the author considers a  family of $q$-Wishart orthogonal matrices.

In this section, we give the same type of results using the same approach as before to study operator-valued Wishart matrices. However, the map $ B: (h_{ij})_{1\leq i,j \leq N} \mapsto B( (h_{ij})_{1\leq i,j \leq N}) = \frac{1}{N} HH^*$ is not linear and thus our noncommutative Lindeberg method in Section \ref{section:Lindeberg} is not directly applicable. To fix this up, we shall rather  consider the Hermitian  matrix
\[
X=\frac{1}{ \sqrt{N} } \left(\begin{array}{cc}
0 & H \\ 
H^* & 0
\end{array} \right) \in M_{2N} (\M) ,
\]
and note that a direct application of the Schur complement formula yields  for any $z\in \mathbb{C}^+$
\begin{equation}\label{CauchyTrans-relation}
(\tau\otimes \tr_{2N}) [\R_{X}(z)]= z \cdot (\tau \otimes \tr_{N}) [\R_{\frac{1}{N} HH^*}(z^2)] \,.
\end{equation}
This reduces the study to the matrix $X$ which is in turn equivalent,  through a permutation of the entries, to the Hermitian matrix $\mathbf{X}=\frac{1}{ \sqrt{N} } (x_{ij}+x_{ji}^*)_{1 \leq i,j \leq N}$ with 
\begin{equation}\label{entries-gram}
x_{ij}= 
\left(\begin{array}{cc}
0 & h_{ij} \\ 
0 & 0
\end{array} \right) \in M_{2} (\M)  \, . 
\end{equation}
Note that $\mathbf{X}$ is an operator-valued Wigner type matrix with entries in $M_2(\M)$. This allows us to apply the results in Section \ref{section:Op-v-matrices} and extend them to Wishart matrices. 
\subsection{Free entries}\label{Section:Wishart-free} We consider first matrices with free centered entries that don't need to be identically distributed. 
\begin{theorem}
 Let $(\M,\tau)$ be a  tracial $W^*$-probability space and  $\mathbf{h}=(h_{ij})_{1\leq i,j \leq N}$ be a family of free centered elements in $\M$ such that $\tau (h_{ij}^*h_{ij})=\tau (h_{ji}^*h_{ji})$. Let $\mathbf{c}=(c_{ij})_{1\leq i,j \leq N}$ be a family of free centered circular elements in $\M$ such that  $\tau (c_{ij}^*c_{ij}) =\tau (h_{ij}^*h_{ij})$ for every $1 \leq i,j \leq N$. Then there exists a positive constant $K$ which depends on the distribution of $h_{11}$ such that for any $z \in \mathbb{C}^+$,
\[
\big| (\tau \otimes \tr_N)[\R_{\frac{1}{N}HH^*} (z)]-(\tau \otimes \tr_N)[\R_{\frac{1}{N}CC^*} (z)] \big| \leq \frac{K}{\Im (z)^4} \frac{1}{\sqrt{N} \, },
\] 
 where $C$ is the $N\times N$ operator valued matrix given by $C= [(c_{ij})_{1 \leq i,j \leq N}]$. 
\end{theorem}

\begin{proof}
As mentioned above, it is equivalent to prove  for any $z \in \mathbb{C}^+$,
\[
\big| (\tau \otimes \tr_{2N})[\R_{X} (z)]-(\tau \otimes \tr_{2N})[\R_{Y} (z)] \big| \leq \frac{K}{\Im (z)^4} \frac{1}{\sqrt{N}} \, ,
\] 
where $X=\frac{1}{ \sqrt{N} } (x_{ij}+x_{ji}^*)_{1 \leq i,j \leq N}$ and $Y=\frac{1}{ \sqrt{N} } (y_{ij}+y_{ji}^*)_{1 \leq i,j \leq N}$  with
\[
x_{ij}= 
\left(\begin{array}{cc}
0 & h_{ij} \\ 
0 & 0
\end{array} \right)
\quad \text{and} \quad
 y_{ij}= 
\left(\begin{array}{cc}
0 & c_{ij} \\ 
0 & 0 \end{array} \right) \in M_2(\M) 
 \,. 
\]
Recalling the linear map $A$ defined in \eqref{def:map_matrix}, we note that $\mathbf{X}= A\big((\sqrt{2} x_{ij})_{1\leq i,j \leq N} \big)$  and  $\mathbf{Y}= A\big((\sqrt{2} y_{ij})_{1\leq i,j  \leq N} \big)$. The proof then follows from Theorem \ref{theo:op-v-free}.
\end{proof}

\begin{remark}
 If  $\tau (c_{ij}c_{ij}^*)=1$ for all $i,j \in \{1, \ldots , N \}$  then $\frac{1}{\sqrt{N}}C$ is itself a circular element in $\M$ and hence $\frac{1}{{N}}CC^*$ is a free Poisson element in $\M$. This can be proved using similar ideas as in \cite[Chapter 6, Theorem 11]{Mingo-Speicher}. 
\end{remark}

\subsection{Exchangeable entries}\label{Section:Wishart-exch}
We study in this section the behavior of operator-valued Wishart-type matrices with exchangeable entries. We state now the main result:

\begin{theorem}\label{thm:Gram-exch}
Let  $(\M, \tau)$ be a tracial $W^*$-probability space with $\M$ being a direct integral of finite dimensional factors of dimension $\leq d$.   Let  $\mathbf{h}=(h_{ij})_{1\leq i,j \leq N} $ be a family of exchangeable variables in $\M^n$ with $n=N^2$.  Set $\bar{h}_{ij}=h_{ij}-\mu$ where $\mu=\frac{1}{n}\sum_{i,j=1}^N h_{ij}$ and consider their variance function 
\[
\eta : \M \rightarrow \M, 
\qquad 
\eta(b)=\frac{1}{n}\sum_{i,j=1}^n \big( \bar{h}_{ij}b\bar{h}_{ij}^*+\bar{h}_{ij}^*b\bar{h}_{ij} \big).
\] 
Let $\mathcal{B}$ be the smallest $W^*$-algebra which is closed under $\eta$ and consider two centered $\B$-valued semicircular variable $S_1$ and $S_2$ of variances $E_{\mathcal{B}}(S_1bS_1)=E_{\mathcal{B}}(S_2bS_2)= \eta(b)$ which are free with amalgamation over $\mathcal{B}$. Setting $C=(S_1+iS_2)/\sqrt{2}$, then there exists a universal constant $c>0$ which depends on the distribution of $h_{11}$ such that  for any $z\in \mathbb{C}^+ $, 
\[
\big|(\tau \otimes \tr_N)[\R_{\frac{1}{N}HH^*}(z)] - \tau[\R_{CC^*}(z)]\big|
\leq c \frac{(|z|\vee 1)^4}{(\Im(z)\wedge 1)^6} \frac{1}{\sqrt{N}}.
\]
\end{theorem}
The variable  $C$ is a $\B$-valued circular variable of variance $E_{\mathcal{B}}(CbC)=E_{\mathcal{B}}(C^*bC^*)=0$, $E_{\mathcal{B}}(CbC^*)=\frac{1}{n}\sum_{i,j=1}^n \bar{h}_{ij}b\bar{h}_{ij}^*$ and $E_{\mathcal{B}}(C^*bC)=\frac{1}{n}\sum_{i,j=1}^n \bar{h}_{ij}^*b\bar{h}_{ij}.$
\begin{proof}Let $\mathbf{x}=(x_{ij})_{1\leq i,j \leq N} $ with  $x_{ij}$ as defined as in \eqref{entries-gram}, set $\bar{x}_{ij}=x_{ij}-\mu$ where $\mu=\frac{1}{n}\sum_{i,j=1}^N x_{ij}$ and consider their variance function 
\[
\eta_2 : M_2(\M) \rightarrow M_2(\M), 
\qquad 
\eta_2(b)=\frac{1}{n}\sum_{i,j=1}^n \big(\bar{x}_{ij}b\bar{x}_{ij}^*+\bar{x}_{ij}^*b\bar{x}_{ij}\big).
\] 
Let $\mathcal{B}_2$ be the smallest $W^*$-algebra which is closed under $\eta_2$ and consider a centered $\B_2$-valued semicircular variable $\mathbf{S}$ of variance $E_{\B_2}(\mathbf{S}b\mathbf{S})= \eta_2(b)$.
Because $\mathbf{X}=A\Big((\sqrt{2}x_{ij})_{1\leq i,j \leq N}\Big)$ with $A$ as defined in \eqref{def:map_matrix}, we can apply Theorem~\ref{thm:Op-v-matrices-exch} and get
\[
\big|(\tau \otimes \tr_{2N})[\R_{\mathbf{X}}(z)] - \tau[\R_{\mathbf{S}}(z)]\big|
\leq c\frac{ (|z|\vee 1)^4}{(\Im(z)\wedge 1)^6}\frac{1}{\sqrt{N}},
\]
for some  universal constant $c>0$. The proof ends by using \eqref{CauchyTrans-relation} and noting that $\mathbf{S}$ has the same distribution as
 \[\left(\begin{array}{cc}
0 & C \\ 
C^* & 0
\end{array} \right) \in M_{2N} (\M) .
\qedhere\]
\end{proof}

 \subsection{Random matrices with correlated  blocks}
 We consider in Theorem \ref{theo:Wishart-Correlated-Gau} matrices with a Wishart block structure in which the blocks can be correlated but their entries are i.i.d. random variables having finite third moments.
This generalizes  ~\cite[Theorem 1]{Ra-Or-Br-Sp-08} in which the limiting distribution was described for Gaussian entries and without any rates of convergence.
\begin{theorem}\label{theo:Wishart-Correlated-Gau} Fix $d\geq 1$ and consider the $Nd \times Nd$  block matrix $H_N$ defined by
$$H_N= \frac{1}{\sqrt{d}}\left(\begin{array}{ccc}
H^{(11)} & \cdots& H^{(1d)} \\ 
\vdots& \ddots&\vdots \\ 
H^{(d1)} & \cdots & H^{(dd)}
\end{array} \right), $$ 
where, for any $1\leq i,j \leq d $,  $H^{(ij)}=(\frac{1}{\sqrt{N}}h_{rp}^{(ij)})_{r,p=1}^N$ are $N\times N$  random matrices having  i.i.d entries and such that  for any $i,j,k,l =1, \ldots, d$ 
$$ 
\text{Cov}(h_{rs}^{(ij)}, \overline{h}_{pq}^{(kl)})=\delta_{rp}\delta_{sq}\sigma(i,j;k,l), \; \text{Cov}(h_{rs}^{(ij)}, h_{rs}^{(kl)})=0\quad \text{and} \quad   \E[|h_{11}^{(ij)}|^3]< \infty,$$
where $\sigma$ is real-valued covariance function such that $\sigma(i,j;k,l)={\sigma(k,l;i,j)}$. Then, for $N\to \infty$, the matrix $H_NH_N^*$ has a limiting eigenvalue distribution whose Cauchy transform $g$ is determined by $g(z)=\tr_d(\R(z))$, where $\R$ is an $M_d(\mathbb{C})$-valued analytic function on $\mathbb{C}^+$, which is uniquely determined by the requirement that, $z\R(z)\to 1$ when $|z|\to \infty$, and that, for $z\in \mathbb{C}^+$, 
$$z\R(z)=1+\eta_1\big(\big(1-\eta_2(\R(z))\big)^{-1}\big) \cdot \R(z),$$
where $\eta_1:M_d(\mathbb{C})\to M_d(\mathbb{C})$ and $\eta_2:M_d(\mathbb{C})\to M_d(\mathbb{C})$  are  the covariance mappings
$$
\eta_1(B)_{ij}=\frac{1}{2d} \sum_{k,l=1}^d\sigma(i,k;j,l)B_{kl}
\quad \text{and} \quad
\eta_2(B)_{ij}=\frac{1}{2d} \sum_{k,l=1}^d\sigma(k,i;l,j)B_{lk}
. $$
Moreover, there exists $c>0$ which depends on the distribution of $( h_{11}^{(ij)})_{i,j=1}^d$ such that for any $z \in \mathbb{C}^+ $,
\[
\big|\mathbb{E}[\tr_{dN}(\R_{H_NH_N^*}(z))] - g(z)\big|
\leq c\frac{ (|z|\vee 1)^4}{(\Im(z)\wedge 1)^6}\frac{1}{\sqrt{N}}.
\]
\end{theorem}

\begin{proof}
Note that $H_N= \frac{1}{\sqrt{d}}\sum_{i,j}H^{(ij)}\otimes E_{ij}
= \frac{1}{\sqrt{N}}\sum_{r,p}E_{rp}\otimes H_{rp}$
where   the blocks $H_{rp}= \frac{1}{\sqrt{d}}( h_{rp}^{(ij)})_{i,j=1}^d$ are i.i.d. $d\times d$   random matrices and hence exchangeable. The study is then reduced to the matrix $X_N  =\frac{1}{ \sqrt{N} } (x_{ij}+x_{ji}^*)_{1 \leq i,j \leq N}$ with  
\[
x_{ij}= 
\left(\begin{array}{cc}
0 & H_{ij} \\ 
0 & 0
\end{array} \right) \in M_{2d} (\M)  \, . 
\]
Then  Theorem \ref{the:i.i.dblocksKroneckerform} holds with $\eta:  M_{2d} (\M) \rightarrow M_{2d} (\M)$ given by 
\[
\eta
 \left(\begin{array}{cc}
B_{11} & B_{12} \\ 
B_{21} & B_{22} 
\end{array} \right) = 
\left(\begin{array}{cc}
\eta_1(B_{22}) & 0 \\ 
0 & \eta_2(B_{11}) 
\end{array} \right),
\]
which is the same covariance map obtained in \cite{Ra-Or-Br-Sp-08} and hence the rest of the arguments follow. 
\end{proof}

\section{Proof of Theorem \ref{theo:Lindeberg-exch}} \label{Section:proof-main-result}

 In the commutative setting, the limiting distribution of a Wigner type matrix,  whose entries are  exchangeable random variables, was shown in \cite{Ch-06} to be semicircular. This was done via the Lindeberg method  that allows approximating the spectral distribution of the matrix by that of a GOE, a Wigner matrix with i.i.d. Gaussian entries, and then concluding its limit. Our proof  is inspired from its commutative analogue  and is also based on the Lindeberg method  together with an operator-valued Gaussian interpolation technique. However, it is not a straightforward extension as one has to deal with all the difficulties arising in the noncommutative realm that need nontrivial manipulation of some terms when exchangeable elements are considered and freeness is dropped. 

We proceed now to the proof and consider the shifted random elements $\bar{x}_1,\dots,\bar{x}_n,\bar{y}_1,\dots,\bar{y}_n$ given  by $\bar{x}_i=x_i-\frac{1}{n}\sum_{k=1}^nx_k$ and $\bar{y}_i=y_i-\frac{1}{n}\sum_{k=1}^ny_k$  and set
$$\bar{\mathbf x} = \sum_{i=1}^n\bar{x}_i\otimes^* a_i\ \ \text{and}\ \ \bar{\mathbf y}= \sum_{i=1}^n\bar{y}_i\otimes^* a_i,$$
where we recall that, for any $x\in \M$ and $a\in \N$,  $x\otimes^* a=x \otimes a +x^*\otimes a^*$. To prove the main result, we shall first reduce our study to the  \emph{exchangeable} elements $(\bar{x}_1,\dots,\bar{x}_n)$ and $(\bar{y}_1,\dots,\bar{y}_n)$  that are such that $\sum_{i=1}^n \bar{x}_i=\sum_{i=1}^n \bar{y}_i=0$.

To do so, we first remark that for any $z \in \mathbb{C}^+$, we have by the resolvent identity  
\begin{align*}
|\tau\otimes \varphi(\R_{\mathbf x}(z)) - \tau\otimes \varphi(\R_{\bar{\mathbf x}}(z))|
&=\big|\tau\otimes \varphi \big(\R_{\mathbf x}(z)(\uX-\bar{\uX}) \R_{{\bar{\uX}}(z)}(z)\big)\big|\\
&\leq \frac{1}{\Im(z)^2} \left\|\left(\frac{1}{n}\sum_{k=1}^n x_k\right)\otimes^* \left(\sum_{i=1}^na_i\right)\right\|_{L^1}\\
&\leq 2 \frac{ K_1}{\Im(z)^2} \Big\|\frac{1}{n}\sum_{i=1}^n x_i\Big\|_{\infty}.
\end{align*}
By exchangeability, we have
$$\Big\|\frac{1}{n}\sum_{k=1}^n x_k\Big\|_{\infty}\leq \frac{1}{n}\sum_{k=1}^n \left\| x_k\right\|_{\infty}=\left\| x_1\right\|_{\infty} .$$
Similarly, we get
$$
\big|\mathbb{E}\big[\varphi(\R_{\mathbf y}(z)) - \varphi(\R_{\bar{\mathbf y}}(z))\big]\big|
\leq
2 \frac{ K_1}{\Im(z)^2}\mathbb{E}\Big\|\frac{1}{n}\sum_{k=1}^n y_k\Big\|_{\infty},
$$
with
$$\frac{1}{n}\sum_{k=1}^n y_k=\frac{1}{n}\sum_{k=1}^n \Big(\frac{1}{\sqrt{n}}\sum_{l=1}^nN_{k,l} \Big) x_l-\frac{1}{n}\sum_{k=1}^n \Big(\frac{1}{\sqrt{n}}\sum_{l=1}^nN_{k,l}\Big) \Big( \frac{1}{n} \sum_{j=1}^n x_j\Big), $$
which can be bounded as follows:
\[
\mathbb{E} \Big\| \frac{1}{n}\sum_{k=1}^n \Big(\frac{1}{\sqrt{n}}\sum_{l=1}^nN_{k,l} \Big) x_l \Big\|_{\infty}
\leq \frac{1}{n}\sum_{l=1}^n \mathbb{E}\Big|\frac{1}{\sqrt{n}}\sum_{k=1}^nN_{k,l}\Big| \|x_1\|_\infty \leq\|x_1\|_\infty,
\]
and 
\[
\mathbb{E} \Big\| \frac{1}{n}\sum_{k=1}^n \Big(\frac{1}{\sqrt{n}}\sum_{l=1}^nN_{k,l}\Big) \Big( \frac{1}{n} \sum_{j=1}^n x_j\Big)\Big\|_{\infty}
\leq \frac{1}{n}\sum_{l=1}^n\mathbb{E}\Big|\frac{1}{\sqrt{n}}\sum_{k=1}^n N_{k,l}\Big| \|x_1\|_\infty \leq\|x_1\|_\infty. 
\]
Thus, we get that
$$\big|\tau\otimes \varphi(\R_{\mathbf x}(z)) - \mathbb{E}[\tau\otimes\varphi(\R_{\mathbf y}(z))] \big|\leq 6 \frac{ K_1}{\Im(z)^2}\|x_1\|_{\infty}+ \big|\tau\otimes \varphi(\R_{\bar{\uX}}(z)) - \mathbb{E}[\tau\otimes\varphi(\R_{\bar{\mathbf y}}(z))] \big|.$$ 
It remains to control the term
\[
\big|\tau\otimes \varphi(\R_{\bar{\uX}}(z)) - \mathbb{E}[\tau\otimes\varphi(\R_{\bar{\mathbf y}}(z))] \big|.
\]
 This is  done in the following Proposition~\ref{theo:exch-zeromean} by which we can conclude our main statement. 

\begin{proposition}\label{theo:exch-zeromean}
Let $(\M,\tau)$ and $(\N,\varphi)$ be two tracial $W^*$-probability spaces. Let $(x_1,\dots,x_n)$ be an $n$-tuple of \emph{exchangeable} elements in $\M$ such that $\sum_{i=1}^nx_i=0$.  Consider a family $(N_{i,k})_{1\leq i,k \leq n}$  of independent standard Gaussian random variables and let $(y_1,\dots,y_n)$ be the $n$-tuple of random elements in $\M$ given by
$$y_i=\frac{1}{\sqrt{n}}\sum_{k=1}^n N_{i,k}x_k$$
and consider the shifted random elements $(\bar{y}_1,\dots,\bar{y}_n)$ given by $$\bar{y}_i=y_i-\frac{1}{n}\sum_{k=1}^ny_k.$$
Let $(a_1,\dots,a_n)$ be an $n$-tuple of elements in $\N$ and set
$$\mathbf x = \sum_{i=1}^nx_i\otimes a_i+x_i^*\otimes a_i^*\ \ \text{and}\ \ \bar{\mathbf y}= \sum_{i=1}^n\bar{y}_i\otimes a_i+\bar{y}_i^*\otimes a_i^*.$$
Then,  for any $z\in \mathbb{C}^+$,
$$|\tau\otimes \varphi(\R_{\uX}(z))  - \mathbb{E}[\tau\otimes \varphi(\R_{\bar{\mathbf y}}(z))]| \leq \frac{K}{(\Im(z)\wedge 1)^5}n^{-1/4},$$
with
$$
K=C\cdot\left( K_2^2\|x_1\|^2_{\infty}n^{3/4} +K_\infty K_2^2\|x_1\|_{\infty}^3n^{5/4}+K_\infty^2K_2^2\|x_1\|^4_{\infty}n^{5/4}\right),$$
where $K_\infty=\max_i \| a_i\|_{\infty}, K_2=\max_i \| a_i\|_{L^2}$, $K_1= \|\sum_{i=1}^n a_i\|_{L^1}$ and $C$ is a universal constant. 
\end{proposition}

\begin{proof}  
Consider $(u_1,\ldots,u_n)$ given by
\[
u_i=x_i+ \frac{1}{n-i+1} \sum_{j<i} x_j=x_i-\frac{1}{n-i+1}\sum_{j\geq i} x_j,
\]
and the random elements $(v_1,\ldots,v_n)$ given by
$$v_i=\frac{1}{\sqrt{n}}\sum_{k=1}^n \tilde{N}_{i,k} x_k,$$
where $(\tilde{N}_{i,k})_{1\leq i,k \leq n}$  is a family of independent standard Gaussian random variables which are independent from $(N_{i,k})_{1\leq i,k \leq n}$.
Let $(b_1,\ldots,b_n)$ be given for each $i= 1, \ldots n$ by 
$$b_i=a_i-\frac{1}{n-i}\sum_{i<j}a_j$$ 
in such a way that
$$\mathbf{x}=\sum_{i=1}^n x_i\otimes^* a_i=\sum_{i=1}^n u_i\otimes^* b_i.$$
Also we set $\mathbf{v}=\sum_{i=1}^n v_i\otimes^* b_i$.
The proof will be done in two major steps: for any $z \in \mathbb{C}^+$,  we first prove via the Lindeberg method in Theorem \ref{thm:Lindeberg}:
\begin{multline}\label{proof:exch_Lindeberg}
\big|\tau \otimes \varphi \big(\R_{\mathbf{x}}(z)\big) - \E \big[\tau \otimes \varphi \big(\R_{\mathbf{v}}(z)\big)\big]\big| 
\\ \leq C\left( \frac{K_2^2}{\Im(z)^3}\|x_1\|^2_{\infty}\sqrt{n} +\frac{K_\infty K_2^2}{\Im(z)^4}\|x_1\|_{\infty}^3n+ \frac{ K_\infty^2K_2^2}{\Im(z)^5}\|x_1\|^4_{\infty}n \right)\, ,
\end{multline}
and then via a Gaussian interpolation technique: 
\begin{equation}\label{proof:exch_interpolation}
\Big|\E \big[\tau \otimes \varphi \big(\R_{\bar{\mathbf{y}}}(z)\big)\big] -\E \big[\tau \otimes \varphi \big(\R_{\mathbf{v}}(z)\big)\big]\Big|
\leq \frac{3K_2^2}{\Im(z)^3}\|x_1\|^2_{\infty}\sqrt{n} \, .
\end{equation}
\paragraph*{\textbf{Application of the Lindeberg method}}
Considering  the $n$-tuples $(u_1,\ldots,u_n)$ and $(v_1,\ldots,v_n)$ in $\M$ and $(b_1,\ldots,b_n)$ in $\mathcal{N}$, we apply  Theorem \ref{thm:Lindeberg} and get, for any $z\in \mathbb{C}^+$,
$$
\big|\tau \otimes \varphi \big(\R_{\mathbf{x}}(z)\big) - \E \big [\tau \otimes \varphi \big(\R_{\mathbf{v}}(z)\big)\big]\big|  \leq \sum_{i=1}^n (|\E[P_i]|+|\E[Q_i]|+|\E[T_i]|),
$$
where in this case $\mathbf z_i $ and $\mathbf z_i^0$ are given by 
$$
\mathbf z_i = u_1\otimes^* b_1 +\cdots + u_i\otimes^* b_i+v_{i+1}\otimes^* b_{i+1}+ \cdots+ v_n\otimes^* b_n,
$$
$$ \mathbf z_i^0 = u_1\otimes^* b_1 +\cdots + u_{i-1}\otimes^* b_{i-1}+v_{i+1}\otimes^* b_{i+1}+ \cdots + v_n\otimes^* b_n,
$$
and the terms $P_i, Q_i$ and $T_i$ are  thus:
\[
P_i=\tau\otimes \varphi\big[\R_{\mathbf z_i^0}(z)(u_i\otimes^* b_i-v_i\otimes^* b_i)\R_{\mathbf z_i^0}(z) \big] ,
\]
$$
Q_i= \tau\otimes \varphi\big[ \R_{\mathbf z_i^0}(z)(u_i\otimes^* b_i)\R_{\mathbf z_i^0}(z)(u_i\otimes^* b_i)\R_{\mathbf z_i^0}(z) \big]  -\tau\otimes \varphi\big[\R_{\mathbf z_i^0}(z)(v_i\otimes^* b_i)\R_{\mathbf z_i^0}(z)(v_i\otimes^* b_i)\R_{\mathbf z_i^0}(z)\big] ,
$$
and
$$
T_i= \tau\otimes \varphi\Big[\R_{\mathbf z_i}(z) \big((u_i \otimes^* b_i)\R_{\mathbf z_i^0}(z) \big)^3
\Big] 
-\tau\otimes \varphi\Big[\R_{\mathbf z_{i-1}}(z) \big((v_i \otimes^* b_i)\R_{\mathbf z_i^0}(z) \big)^3
\Big] .
$$
Now to control the above quantities, it is convenient to  introduce first some notation:  for all $i,j,k\in \{1,\ldots, n\}$, let $u_i^{j\leftrightarrow k}$ and $v_i^{j\leftrightarrow k}$  denote the elements defined exactly as $u_i$ and $v_i$ but by swapping $x_j$ and $x_k$. We also set
$$\mathbf z_i^{0,j\leftrightarrow k} = u_1^{j\leftrightarrow k}\otimes^* b_1 +\cdots + u_{i-1}^{j\leftrightarrow k}\otimes^* b_{i-1}+v_{i+1}^{j\leftrightarrow k}\otimes^* b_{i+1}+ \cdots + v_n^{j\leftrightarrow k}\otimes^* b_n.$$
A key point of the proof is the following lemma which follows from the exchangeability of the $x_i$'s.
\begin{lemma}\label{lem:exch-keypoint}
Let $i\leq j$ and $i\leq k$ then for any polynomial $P$ in three noncommuting variables, we have
$$\E\big[\tau\otimes \varphi[P(\R_{\mathbf z_i^0}(z),x_j\otimes^* b_i,x_k\otimes^* b_i)]\big]=\E\big[\tau\otimes \varphi[P(\R_{\mathbf z_i^0}(z),x_k\otimes^* b_i,x_j\otimes^* b_i)]\big].$$
\end{lemma}

\begin{proof}[Proof of the Lemma \ref{lem:exch-keypoint}] Note that $(v_1,\ldots,v_n)$ has the same law as $(v_1^{j\leftrightarrow k},\ldots,v_n^{j\leftrightarrow k})$ as a random variable, and thus, if $i\leq j$ and $i\leq k$, then $\mathbf z_i^{0,j\leftrightarrow k}$ has the same law as $\mathbf z_i^{0}$ as a random variable. As a consequence, we just need to prove that
$$\tau\otimes \varphi[P(\R_{\mathbf z_i^{0}}(z),x_j\otimes^* b_i,x_k\otimes^* b_i)]=\tau\otimes \varphi[P(\R_{\mathbf z_i^{0,j\leftrightarrow k}}(z),x_k\otimes^* b_i,x_j\otimes^* b_i)].$$
This equality, valid for all $i,j,k\in \{1,\ldots, n\}$, is a direct consequence of exchangeability.
Indeed, for any polynomial $Q$, we have
$$\tau[Q(u_1,\ldots, u_{i-1},v_{i+1},\ldots, v_n,x_j,x_k)]=\tau[Q(u_1,\ldots, u_{i-1},v_{i+1}^{j\leftrightarrow k},\ldots,v_n^{j\leftrightarrow k},x_k,x_j)],$$
from which we deduce that,
for any polynomial $Q$, we have
\begin{multline*} \tau\otimes \varphi[Q(u_1\otimes^* b_1,\ldots, u_{i-1}\otimes^* b_{i-1},v_{i+1}\otimes^* b_{i+1},\ldots, v_n\otimes^* b_{n},x_j\otimes^* b_{i},x_k\otimes^* b_{i})] \\= \tau \otimes \varphi [Q(u_1\otimes^* b_{1},\ldots, u_{i-1}\otimes^* b_{i-1},v_{i+1}^{j\leftrightarrow k}\otimes^* b_{i+1},\ldots,v_n^{j\leftrightarrow k}\otimes^* b_{n},x_k\otimes^* b_{i},x_j\otimes^* b_{i})].
\end{multline*}
 Finally, by an approximation procedure, we get
\[\tau\otimes \varphi[P(\R_{\mathbf z_i^{0}}(z),x_j\otimes^* b_i,x_k\otimes^* b_i)]=\tau\otimes \varphi[P(\R_{\mathbf z_i^{0,j\leftrightarrow k}}(z),x_k\otimes^* b_i,x_j\otimes^* b_i)].\qedhere\]
\end{proof}
By Lemma \ref{lem:exch-keypoint}, we have for $j\geq i$,
$$\mathbb{E}\Big[\tau\otimes \varphi\big[\R_{\mathbf z_i^0}(z)(x_i\otimes^* b_i)\R_{\mathbf z_i^0}(z) \big]\Big]=\mathbb{E}\Big[\tau\otimes \varphi\big[\R_{\mathbf z_i^0}(z)(x_j\otimes^* b_i)\R_{\mathbf z_i^0}(z) \big]\Big],$$
which implies that
$$\mathbb{E}\Big[\tau\otimes \varphi\big[\R_{\mathbf z_i^0}(z)(u_i\otimes^* b_i)\R_{\mathbf z_i^0}(z) \big]\Big]= 0.$$ On the other hand, we note that $\mathbb{E}[v_i\otimes^* b_i]=0$ and $v_i\otimes^* b_i$ is independent from $\mathbf z_i^0$, which imply that
$$\mathbb{E}\Big[\tau\otimes \varphi\big[\R_{\mathbf z_i^0}(z)(v_i\otimes^* b_i)\R_{\mathbf z_i^0}(z) \big]\Big]= 0.$$
As a consequence,  $\mathbb{E}[P_i]=0$. Now turning to $Q_i$ and again by  using Lemma  \ref{lem:exch-keypoint}, we get  for any  $l \geq i$
\begin{align*}
&\E\Big[\tau\otimes \varphi \big(\R_{\mathbf z_i^0}(z)(x_i\otimes^* b_i)\R_{\mathbf z_i^0}(z)\Big(\frac{1}{n-i+1}\sum_{j\geq i}x_j\otimes^* b_i \Big)\R_{\mathbf z_i^0}(z)\big) \Big]\\
&=\E\Big[\tau\otimes \varphi \big(\R_{\mathbf z_i^0}(z)(x_l\otimes^* b_i)\R_{\mathbf z_i^0}(z)\Big(\frac{1}{n-i+1}\sum_{j\geq i}x_j\otimes^* b_i \Big)\R_{\mathbf z_i^0}(z)\big) \Big]\\
&= 
\E\Big[\tau\otimes \varphi \big(\R_{\mathbf z_i^0}(z)\Big(\frac{1}{n-i+1}\sum_{j\geq i}x_j\otimes^* b_i \Big)\R_{\mathbf z_i^0}(z)\Big(\frac{1}{n-i+1}\sum_{j\geq i}x_j\otimes^* b_i \Big)\R_{\mathbf z_i^0}(z)\big) \Big],
\end{align*}
and thus we get
\begin{align*}
&\E\Big[\tau\otimes \varphi \big(\R_{\mathbf z_i^0}(z)(u_i\otimes^* b_i)\R_{\mathbf z_i^0}(z)(u_i\otimes^* b_i)\R_{\mathbf z_i^0}(z)\big) \Big]\\
=&\E\Big[\tau\otimes \varphi \big(\R_{\mathbf z_i^0}(z)(x_i\otimes^* b_i)\R_{\mathbf z_i^0}(z)(x_i\otimes^* b_i)\R_{\mathbf z_i^0}(z)\big) \Big]\\
&-\E\Big[\tau\otimes \varphi \big(\R_{\mathbf z_i^0}(z)\Big(\frac{1}{n-i+1}\sum_{j\geq i}x_j\otimes^* b_i \Big)\R_{\mathbf z_i^0}(z)\Big(\frac{1}{n-i+1}\sum_{j\geq i}x_j\otimes^* b_i \Big)\R_{\mathbf z_i^0}(z)\big) \Big].
\end{align*}
On the other hand, we have
\begin{align*}
\E & \Big[\tau\otimes \varphi \big(\R_{\mathbf z_i^0}(z)(v_i\otimes^* b_i)\R_{\mathbf z_i^0}(z)(v_i\otimes^* b_i)\R_{\mathbf z_i^0}(z)\big) \Big]
\\=&
\frac{1}{n}\sum_{k,l} \E (\tilde{N}_{i,k} \tilde{N}_{i,l}) \E\Big[\tau\otimes \varphi \big(\R_{\mathbf z_i^0}(z)(x_k\otimes^* b_i)\R_{\mathbf z_i^0}(z)(x_l\otimes^* b_i)\R_{\mathbf z_i^0}(z)\big) \Big]
\\=&\E \Big[\frac{1}{n} \sum_k \tau\otimes \varphi \big(\R_{\mathbf z_i^0}(z)(x_k\otimes^* b_i)\R_{\mathbf z_i^0}(z)(x_k\otimes^* b_i)\R_{\mathbf z_i^0}(z)\big)
\\
=&\E \Big[ \frac{1}{n}\sum_{k<i}\tau\otimes \varphi \big(\R_{\mathbf z_i^0}(z)(x_k\otimes^* b_i)\R_{\mathbf z_i^0}(z)(x_k \otimes^* b_i)\R_{\mathbf z_i^0}(z)\big)\Big]
\\&+
\E \Big[ \frac{n-i+1}{n}\tau\otimes \varphi \big(\R_{\mathbf z_i^0}(z)(x_i\otimes^* b_i)\R_{\mathbf z_i^0}(z)(x_i\otimes^* b_i)\R_{\mathbf z_i^0}(z)\big)\Big].
\end{align*}
As a consequence,   $|\mathbb{E}[Q_i]|$  is bounded by the terms
\begin{align*}
& \bigg|\E \Big[\frac{i-1}{n}\tau\otimes \varphi \big(\R_{\mathbf z_i^0}(z)(x_i\otimes^* b_i)\R_{\mathbf z_i^0}(z)(x_i\otimes^* b_i)\R_{\mathbf z_i^0}(z)\big)\Big]
\\ & \qquad \qquad
 -\E \Big[ \frac{1}{n}\sum_{k<i}\tau\otimes \varphi \big(\R_{\mathbf z_i^0}(z)(x_k\otimes^* b_i)\R_{\mathbf z_i^0}(z)(x_k \otimes^* b_i)\R_{\mathbf z_i^0}(z)\big)\Big] \bigg|
\\&+
\Big|\E\Big[\tau\otimes \varphi \big(\R_{\mathbf z_i^0}(z)\Big(\frac{1}{n-i+1}\sum_{j\geq i}x_j\otimes^* b_i \Big)\R_{\mathbf z_i^0}(z)\Big(\frac{1}{n-i+1}\sum_{j\geq i}x_j\otimes^* b_i \Big)\R_{\mathbf z_i^0}(z)\big) \Big]\Big|
\\ & :=| \E [ Q_i^{(1)}] | + | \E [Q_i^{(2)}]|.
\end{align*}
As in the proof of Lemma \ref{lem:exch-keypoint}, one can prove for any $k<i$,
\begin{multline*}
\tau\otimes \varphi \big[\R_{\mathbf z_i^0}(z)(x_i\otimes^* b_i)\R_{\mathbf z_i^0}(z)(x_i\otimes^* b_i)\R_{\mathbf z_i^0}(z)\big]
\\ =
\tau\otimes \varphi \big[\R_{\mathbf z_i^{0,i\leftrightarrow k}}(z)(x_k\otimes^* b_i)\R_{\mathbf z_i^{0,i\leftrightarrow k}}(z)(x_k\otimes^* b_i)\R_{\mathbf z_i^{0,i \leftrightarrow k}}(z)\big].
\end{multline*}
Using the resolvent identity in Lemma \ref{lem:Taylor_resolvents} with $m=1$, we  compute for any $k<i$,
\begin{align*}
&\Big|\tau\otimes \varphi \big[\R_{\mathbf z_i^{0,i\leftrightarrow k}}(z)(x_k\otimes^* b_i)\R_{\mathbf z_i^{0,i\leftrightarrow k}}(z)(x_k\otimes^* b_i)\R_{\mathbf z_i^{0,i\leftrightarrow k}}(z)\big]
\\& \qquad \qquad\qquad -
\tau\otimes \varphi \big[\R_{\mathbf z_i^{0}}(z)(x_k\otimes^* b_i)\R_{\mathbf z_i^{0}}(z)(x_k\otimes^* b_i)\R_{\mathbf z_i^{0}}(z)\big]\Big|\\
\leq &\frac{3}{\Im(z)^2}\|\R_{\mathbf z_i^{0,i\leftrightarrow k}}(z)-\R_{\mathbf z_i^{0}}(z)\|_{\infty}\|x_k\otimes^* b_i\|_{L^2}^2\\
\leq & \frac{48}{\Im(z)^4}\|\mathbf z_i^{0,i\leftrightarrow k}-\mathbf z_i^{0}\|_{\infty}K_2^2\|x_1\|_{\infty}^2.
 \end{align*}
Note that, for all $j\in\{1,\ldots, n\}$, $\|u_j\|_\infty\leq 2\|x_1\|_\infty$ and $\|b_j\|_\infty\leq 2K_\infty$. As a consequence,
\begin{align*}
\|\mathbf z_i^{0,i\leftrightarrow k}-\mathbf z_i^{0}\|_{\infty}
&=\|u_k\otimes^* b_i+u_i\otimes^* b_k-u_k\otimes^* b_k-u_i\otimes^* b_i\|_{\infty}
\\& \leq 32\|x_1\|_{\infty}K_\infty,
\end{align*}
and thus
$$| \E [Q_i^{(1)}] | \leq C\frac{K_\infty K_2^2}{\Im(z)^4}\|x_1\|_{\infty}^3,$$
whenever $C\geq 1536$. On the other hand, turning to the second term $\E [ Q_i^{(2)}] $, we get
\begin{align*}
| \E [ Q_i^{(2)} ] |
\leq  \frac{1}{\Im(z)^3} \Big\|\frac{1}{n-i+1}\sum_{j\geq i}x_j\otimes^* b_i  \Big\|^2_{L^2}
\leq 16 \frac{K_2^2}{\Im(z)^3}\Big\|\frac{1}{n-i+1}\sum_{j\geq i}x_j \Big\|^2_{L^2} .
\end{align*}
Using exchangeability and the same computation as in  Chatterjee \cite[eq. (4)]{Ch-06}, we get
\[
\Big\|
\frac{1}{n-i+1}\sum_{j\geq i}x_j \Big\|^2_{L^2}
=\frac{i-1}{(n-i+1)(n-1)} \|x_1\|^2_{L^2}\leq \frac{1}{n-i+1}\|x_1\|^2_{L^2}.
\]
Collecting all the above bounds, we finally get
\[
|\mathbb{E}[Q_i]|\leq C\frac{K_2^2}{\Im(z)^3} \frac{1}{n-i+1}\|x_1\|^2_{L^2}+C\frac{K_\infty K_2^2}{\Im(z)^4}\|x_1\|_{\infty}^3\, ,
\]
where $C$ is a sufficiently large constant. We control now the third order term $T_i$. Considering its first term, we get the following bound: 
\begin{align*}
\Big|\E \Big[\tau\otimes \varphi\big[\R_{\mathbf z_i}(z) \big((u_i \otimes^* b_i)\R_{\mathbf z_i^0}(z) \big)^3
\big] \Big]\Big|
\leq &\frac{1}{\Im(z)^4} \|u_i \otimes^* b_i\|_{\infty}\|u_i \otimes^* b_i\|_{L^2}^2\\
\leq & 64 \frac{ K_\infty K_2^2}{\Im(z)^4} \|u_i\|_{\infty}\|u_i\|_{L^2}^2\\
\leq & 512  \frac{ K_\infty K_2^2}{\Im(z)^4} \|x_1\|_{\infty}^3.
\end{align*}
Recall the second term in $T_i$:
$\tau\otimes \varphi\Big[\R_{\mathbf z_{i-1}}(z) \big((v_i \otimes^* b_i)\R_{\mathbf z_i^0}(z) \big)^3\Big].$
Developing the last $v_i=\frac{1}{\sqrt{n}}\sum_k \tilde{N}_{i,k} x_k$ in it, we get $n$ terms of the form
\begin{align*}
\frac{1}{\sqrt{n}}  &\Big| \E \Big[  \tilde{N}_{i,k}\tau\otimes \varphi\big[\R_{\mathbf z_{i-1}}(z) \big((v_i \otimes^* b_i)\R_{\mathbf z_i^0}(z) \big)^2(x_k \otimes^* b_i)\R_{\mathbf z_i^0}(z) \big]\Big]\Big|\\
=&
\frac{1}{n} \Big| \E \Big[\tau\otimes \varphi\big[\R_{\mathbf z_{i-1}}(z)(v_i \otimes^* b_i)\R_{\mathbf z_i^0}(z) (x_k \otimes^* b_i)\R_{\mathbf z_i^0}(z) (x_k \otimes^* b_i)\R_{\mathbf z_i^0}(z) \big]\Big]\Big|\\
&+\frac{1}{n} \Big| \E[\tau\otimes \varphi\big[\R_{\mathbf z_{i-1}}(z) (x_k \otimes^* b_i)\R_{\mathbf z_i^0}(z) (v_i \otimes^* b_i)\R_{\mathbf z_i^0}(z)(x_k \otimes^* b_i)\R_{\mathbf z_i^0}(z) \big]\Big]\Big|\\ 
  & \quad + \frac{1}{n} \Big|\E\Big[\tau\otimes \varphi\big[\R_{\mathbf z_{i-1}}(z) (x_k \otimes^* b_i)\R_{\mathbf z_{i-1}}(z) (v_i \otimes^* b_i)\R_{\mathbf z_i^0}(z) (v_i \otimes^* b_i)\R_{\mathbf z_i^0}(z) (x_k \otimes^* b_i)\R_{\mathbf z_i^0}(z) \big]\Big]\Big|,
\end{align*}
where the above equality is obtained via an integration by parts. The third term is bounded by
$$
 \frac{1}{n} \frac{1}{\Im(z)^5} \|x_k\otimes^* b_i\|^2_{\infty} \E \|v_i\otimes^* b_i\|^2_{L^2} 
\leq  C\frac{1}{n} \frac{K_\infty^2K_2^2}{\Im(z)^5} \|x_1\|_{\infty}^2\E\|v_i\|_{L^2}^2\leq C\frac{1}{n} \frac{K_\infty^2K_2^2}{\Im(z)^5} \|x_1\|_{\infty}^4,
$$
since
\[
\E \|v_i\|^2_{L^2}=\frac{1}{n}\sum_{k,l} \E[ \tilde{N}_{i,k}\tilde{N}_{i,l} ] \tau(x_kx_l) 
=\frac{1}{n}\sum_{k}\|x_1\|^2_{L^2}
=\|x_1\|^2_{L^2}.
\]
Similarly, the first and second terms are bounded by $$C\frac{1}{n} \frac{K_\infty K_2^2}{\Im(z)^4} \|x_1\|_{\infty}^3.$$
Collecting the above bounds we finally get,
\[
|\mathbb{E}[T_i]|\leq C\frac{K_\infty K_2^2}{\Im(z)^4} \|x_1\|_{\infty}^3+C\frac{K_\infty^2K_2^2}{\Im(z)^5} \|x_1\|_{\infty}^4.
\]
Noting that $ \sum_{i=1}^n\frac{1}{n-i+1}\leq C\sqrt{n}$ whenever $C$ is sufficiently large, then gathering the bounds on $Q_i$ and $T_i$ and  summing over $i= 1 , \ldots , n$, we end the proof of \eqref{proof:exch_Lindeberg}.

\paragraph*{\textbf{Gaussian Interpolation}}

To prove \eqref{proof:exch_interpolation}, we set for each $0\leq t \leq 1$, 
$
\mathbf w_t=\sqrt{1-t}\mathbf v+\sqrt{t}\bar{\mathbf y}.
$
Then
\begin{align*}
\E \big[\tau\otimes \varphi[\R_{\bar{\mathbf y}}(z)]\big]-\E \big[\tau\otimes \varphi[\R_{\mathbf v}(z)]\big]
&= \E\int_0^1 \partial_t \tau\otimes \varphi[\R_{\mathbf w_t}(z)] \, \text{dt} \\
&=\E \int_0^1  \tau\otimes \varphi\left[\R_{\mathbf w_t}(z)^2
\left(\frac{\bar{\mathbf y}}{2\sqrt{t}} - \frac{\mathbf v}{2\sqrt{1-t}} \right)\right]\text{dt} .
\end{align*}
We can decompose by linearity
\[
\bar{\mathbf y}=\frac{1}{\sqrt{n}}\sum_{i,k=1}^n \Big(N_{i,k} -\frac{1}{n}\sum_{j=1}^n N_{j,k} \Big)x_k\otimes^* a_i
:=\frac{1}{\sqrt{n}}\sum_{i,k=1}^n W_{i,k}x_k\otimes^* a_i
\]
and
\[
\mathbf v=\frac{1}{\sqrt{n}}\sum_{i,k=1}^n \Big(\tilde{N}_{i,k} - \sum_{j=1}^{i-1}\frac{1}{n-j} \tilde{N}_{j,k} \Big)x_k\otimes^* a_i
:=\frac{1}{\sqrt{n}}\sum_{i,k=1}^n \tilde{W}_{i,k} x_k\otimes^* a_i\, .
\]
By integration by parts, we get
\begin{align*}
\E\big[ \tau\otimes \varphi \big(\R_{\mathbf w_t}(z)^2\bar{\uY}\big)\big]
&=\frac{1}{\sqrt{n}}\sum_{i,k=1}^n\E[ W_{i,k} \, \tau\otimes \varphi \big(  \R_{\mathbf w_t}(z)^2(x_k\otimes^* a_i)\big)\big] \\
&= 2\sqrt{t}\sum_{i,i'=1}^n\sigma_{i,i'}\frac{1}{n}\sum_{k=1}^n \E \big[ \tau\otimes \varphi \big(\R_{\mathbf w_t}(z)(x_k\otimes^* a_{i'})\R_{\mathbf w_t}(z)^2(x_k\otimes^* a_i)\big)\big],
\end{align*}
where for any $k \in \{1, \ldots , n \}$,
\[
\sigma_{i,i'}:= \E [W_{i,k} W_{i',k}]
=\left\{ \begin{array}
[c]{ll}%
-n^{-1} & \mbox{ if }i\neq i',\\
(n-1)/{n} & \mbox{ if }i = i' \, .
\end{array} 
\right.  \]
Similarly, we get
\begin{align*}
\E \big[ \tau\otimes \varphi \big(\R_{\mathbf w_t}(z)^2\mathbf v\big)\big]&= \frac{1}{\sqrt{n}} \sum_{i,k=1}^n\E \big[ \tilde{W}_{i,k}\,  \tau\otimes \varphi \big( \R_{\mathbf w_t}(z)^2 (x_k \otimes^* a_i )\big)] \\
&=2\sqrt{1-t}\sum_{i,i'=1}^n\tilde{\sigma}_{i,i'}\frac{1}{n}\sum_{k=1}^n
\E \big[ \tau\otimes \varphi \big(\R_{\mathbf w_t}(z)(x_k\otimes^* a_{i'})\R_{\mathbf w_t}(z)^2(x_k\otimes^* a_i)\big)\big],
\end{align*}
where for any $k \in \{1, \ldots , n \}$,
\[
\tilde{\sigma}_{i,i'} := \E [\tilde{W}_{i,k} \tilde{W}_{i',k}]
=\left\{ \begin{array}
[c]{ll}%
-\dfrac{1}{n-1}-\displaystyle\sum_{j=1}^{i-1}\dfrac{1}{(n-j)^2(n-j-1)} & \mbox{ if }i< i',\\
1+\displaystyle\sum_{j=1}^{i-1}\dfrac{1}{(n-j)^{2}} & \mbox{ if }i = i' ,\\
\tilde{\sigma}_{i',i}  & \mbox{ if } i>i'\, .
\end{array} 
\right. 
 \]
Combining the above quantities, we get
\begin{align*}
&\Big|\E \big[ \tau \otimes \varphi \big(\R_{\bar{\mathbf{y}}}(z)\big)\big] -\E \big[ \tau \otimes \varphi \big(\R_{\mathbf{v}}(z)\big)\big]\Big|\\
=& \sum_{i,i'=1}^n |\sigma_{i,i'}-\tilde{\sigma}_{i,i'}|
 \frac{1}{n}\sum_{k=1}^n \int_0^1\E \big| \tau \otimes \varphi \big(\R_{\mathbf w_t}(z)(x_k\otimes^* a_{i'})\R_{\mathbf w_t}(z)^2(x_k\otimes^* a_{i})\big)\big| \text{dt} \\
\leq& 4 \sum_{i,i'=1}^n |\sigma_{i,i'}-\tilde{\sigma}_{i,i'}|  \frac{K_2^2}{\Im(z)^3} \|x_1\|^2_{\infty}\\
\leq&12  \frac{K_2^2}{\Im(z)^3}\|x_1\|^2_{\infty}\sqrt{n},
\end{align*}
where the last inequality follows by the same computation as in  Chatterjee \cite[eq. (12)]{Ch-06} to show that
\[
 \sum_{i,i'=1}^n|\sigma_{i,i'}-\tilde{\sigma}_{i,i'}|
 \leq 3+\sum_{k=2}^{n-1}\frac{1}{k}
 \leq 3\sqrt{n}.
\qedhere \]
\end{proof}

\end{document}